\documentclass[bj]{imsart}

\arxiv{math.PR/0000000}

\RequirePackage[OT1]{fontenc}
\RequirePackage{amsthm,amsmath}

\usepackage{amssymb,latexsym,amsbsy,amsxtra,amsgen}
\usepackage[ansinew]{inputenc}
\usepackage{comment}
\usepackage{eurosym}
\usepackage{graphicx}
\usepackage{subfig}
\usepackage{caption}
\usepackage{everysel}
\usepackage{keyval}
\usepackage{ragged2e}
\usepackage{float}
\usepackage{multirow}
\usepackage[colorlinks=true,citecolor=blue, urlcolor=blue, linkcolor = blue]{hyperref}

 \newtheorem{thm}{Theorem}[section]
 \newtheorem{cor}[thm]{Corollary}
 \newtheorem{lem}[thm]{Lemma}

 \theoremstyle{definition}
 
 \newtheorem{rem}[thm]{Remark}
 \newtheorem{ex}[thm]{Example}
 \newtheorem{ass}[thm]{Assumption}

\numberwithin{equation}{section}

 \DeclareMathOperator{\tr}{tr}

 \newcommand{\ind}{\mathbf{1}}

 \newcommand{\R}{\mathbb{R}}
 \newcommand{\Z}{\mathbb{Z}}

 \newcommand{\ic}{\mathrm{i}}

 \newcommand{\E}{\mathbb{E}}

 \newcommand{\N}{\mathbb{N}}

 \newcommand{\F}{\mathcal{F}}
 \newcommand{\OO}{\mathcal{O}}

 \newcommand{\oo}{\mbox{\scriptsize $\mathcal{O}$}}

 \newcommand{\ad}{\mathfrak{a}}

 \newcommand{\bs}[1]{\boldsymbol{#1}}

 \newcommand{\HH}{\mathbb{H}}
 \newcommand{\HHS}{\bs{\mathcal{H}}}
 \newcommand{\HSN}{\bs{\mathcal{S}}}
 \newcommand{\rd}{\mathfrak{r}}
 \newcommand{\Ln}{\mathbb{L}}

\allowdisplaybreaks[3]

\begin{document}
\begin{frontmatter}

\title{Rate of convergence for Hilbert space valued processes}
\runtitle{Rate of convergence for Hilbert space valued processes}

\begin{aug}
\author{\fnms{Moritz} \snm{Jirak}\thanksref{a,t1}},

\thankstext{t1}{FOR 1735 {\it Structural Inference in Statistics: Adaptation and Efficiency} is gratefully acknowledged.}

\address[a]{
Institut f\"{u}r Mathematische Stochastik, Technische Universit\"{a}t Braunschweig,
Pockelsstr. 14, 38106 Braunschweig,
Germany, m.jirak@tu-braunschweig.de}

\runauthor{Jirak}
\affiliation{Braunschweig University of Technology}
\end{aug}

\begin{abstract}
Consider a stationary, linear Hilbert space valued process. We establish Berry-Essen type results with optimal convergence rates under sharp dependence conditions on the underlying coefficient sequence of the linear operators. The case of non-linear Bernoulli-shift sequences is also considered. If the sequence is $m$-dependent, the optimal rate $(n/m)^{1/2}$ is reached. If the sequence is weakly geometrically dependent, the rate $(n/\log n)^{1/2}$ is obtained.
\end{abstract}


\end{frontmatter}

\section{Introduction}

Let $\bigl\{X_k\bigr\}_{k \in \Z}$ be a zero mean process takeing values in a separable Hilbert space $\HH$ with inner product $\langle \cdot , \cdot \rangle$ and norm $\|\cdot\|_{\HH}$. A fundamental issue in probability theory and statistics is whether or not the central limit theorem holds for the partial sum $S_n(X) = \sum_{k = 1}^n X_k$, that is, if we have
\begin{align}\label{eq_clt}
\frac{1}{\sqrt{n}} S_n(X) \xrightarrow{w} Z_{\bf \Lambda},
\end{align}
where $Z_{\bf \Lambda}$ denotes a centered Gaussian random variable with associated covariance operator
\begin{align*}
{\bf \Lambda}(\cdot) = \E\bigl[\langle Z_{\bf \Lambda}, \cdot \rangle Z_{\bf \Lambda} \bigr].
\end{align*}
Going one step further, we can ask ourselves about a possible rate of convergence in \eqref{eq_clt}, more precisely, if it holds that
\begin{align}\label{eq_clt_rate}
\lim_{n \to \infty}d\bigl(P_{S_n(X)/\sqrt{n}}, P_{Z_{\bf \Lambda}}\bigr) \, \rd_n  < \infty \quad \text{for a sequence $\rd_n \to \infty$,}
\end{align}
where $d\bigl(\cdot,\cdot\bigr)$ is a probability metric, and $P_{X}$ denotes the probability measure induced by the random variable $X$. The rate $\rd_n$ can be considered as a measure of reliability for statistical inference based on $S_n(X)$, and large rates are naturally preferred. In the context of general Hilbert space valued processes, the notion of 'probability of hitting a ball' has turned out to be a convenient formulation. More precisely, we consider the uniform metric over Balls, that is,
\begin{align}
\Delta_n(\mu) = \sup_{x \in \R}\bigl|P\bigl(\|n^{-1/2}S_n(X) + \mu\|_{\HH} \leq x \bigr) - P\bigl(\|Z_{{\bf \Lambda}} + \mu\|_{\HH} \leq x \bigr) \bigr|, \quad \mu \in \HH,
\end{align}
where $Z_{\bf \Lambda}$ is a zero mean Gaussian random variable with associated covariance operator ${\bf \Lambda}$. If $\bigl\{X_k\bigr\}_{k \in \Z}$ is IID and real valued ($\HH = \R$), a huge literature has evolved around \eqref{eq_clt_rate} in the past decades, see for instance ~\cite{petrov_book_1995}. Interestingly, if $X_k$ lies in a general infinite dimensional Hilbert space $\HH$, much less can be found in the literature. To some extent, this can certainly be attributed to the significantly higher complexity of the problem. While the first optimal results about the rate of convergence in real valued cases appeared around 1940 (cf. ~\cite{berry},~\cite{esseen_1945}), it took more than another $30$ years until analogue results were obtained if $\HH$ is a general, infinite dimensional Hilbert space. Notable contributions here among others are ~\cite{bentkus_1984}, ~\cite{goetze_1979_mmises}, ~\cite{Nagaev_1984},~\cite{Nagaev_2005}, ~\cite{ulyanov_1986} and ~\cite{yurinskii_1982}. For a more detailed account on the historic development, see ~\cite{zalesski_1988}. More recently, weakly dependent Hilbert space valued process have attracted more attention in the statistical context of functional principal component analysis, see ~\cite{hoermann_2010} and ~\cite{horvath_kokoszka_book_2012}. In this note, we are concerned with possibly dependent, stationary processes that can be represented as
\begin{align}
X_k = g_k\bigl(\{\epsilon_j\}_{j \in \Z}\bigr), \quad k \in \Z,
\end{align}
for measurable functions $g_k$ and IID random variables $\{\epsilon_k\}_{k \in \Z} \in \mathbb{S}$ for some measure space $\mathbb{S}$. Such processes are often also referred to as (non-causal) Bernoulli-shift processes. Special emphasis is devoted to non-causal linear processes, that is, we assume that $X_k$ can be represented as
\begin{align}
X_k = \sum_{j \in \Z} \alpha_j(\epsilon_{k + j}), \quad k \in \Z,
\end{align}
where $\bigl\{\epsilon_k\bigr\}_{k \in \Z} \in \HH$ is a centered IID sequence with $\E\bigl[\|\epsilon_k\|_{\HH}^2\bigr] < \infty$. Note that this implies existence of the associated covariance operator ${\bf C}^{\epsilon}$. The sequence $\bigl\{\alpha_j\bigr\}_{j \in \N}$ denotes linear operators, which we endow with the usual operator norm
\begin{align*}
\bigl\|\alpha_j\bigr\|_{\HHS} = \sup_{x \in\HH:\, \|x\|_{\HH} = 1} \bigl\|\alpha_j(x)\bigr\|_{\HH}.
\end{align*}
For notational convenience, we assume here that $\alpha_j$ maps from $\HH$ to $\HH$, but also two different Hilbert spaces are possible. Linear processes are among the first (possibly weak dependent) generalizations from the IID case, but already constitute a relevant class of processes which contains important examples from the time series literature, for instance (functional) autoregressive processes (cf. ~\cite{bosq_2000}, ~\cite{hoermann_2010}). The CLT for linear processes in Hilbert spaces was investigated, among others, in ~\cite{merlevede_1997_opt}, where it was shown that a CLT is valid if and only if
\begin{align}\label{eq_alpha_sharp}
\sum_{j \in \Z}\|\alpha_j\|_{\HHS} < \infty,
\end{align}
see below for some more comments on this result. It seems that the first results about the rate of convergence for linear processes were considered in ~\cite{bosq_2000}, where the special case of Hilbert space valued AR(1) processes was treated, and a rate of $\sqrt{n}$ was reached. Some extensions with possible suboptimal rates are obtained in ~\cite{bosq_2003}, see also the correction in ~\cite{bosq_erratum_2004}. In ~\cite{el_machkouri_2010}, the rate $\sqrt{n}$ was obtained if $\sum_{j \in \Z} |j| \|\alpha_j\|_{\HHS} < \infty$ and the sequence $\bigl\{\epsilon_k\bigr\}_{k \in \N}$ has bounded support, that is, $P\bigl(\|\epsilon_k\|_{\HH} > C \bigr) = 0$ for some $C > 0$. More recently, ~\cite{paulauskas_2011} considered random fields in Hilbert and Banach spaces. In the special case of real-valued sequences $\alpha_j \in \R$, Berry-Esseen type bounds are established if \eqref{eq_alpha_sharp} holds. However, unlike to the previous results, the approximating Gaussian measure depends on $n$ in general, which is different from our results.
\\
Regarding non-linear sequences, the problem becomes more difficult. Certain martingale difference sequences in Banach spaces have been investigated in ~\cite{basu_1988}, ~\cite{butzer_1983}. In ~\cite{rhee_talagrand_1981} and ~\cite{rhee_1986} (see also ~\cite{bjarmotas_paulauskas_1979}), $m$-dependent sequences in Banach spaces are studied, whereas ~\cite{Zuparov_1983} considers $\varphi(n)$-mixing sequences with geometric decay. Though some of these results are optimal or close to optimality in a certain way, they lead to (significantly) inferior rates for Hilbert space valued sequences, as was pointed out in ~\cite{tikhomirov_1991_BE_Hilbert}. ~\cite{tikhomirov_1991_BE_Hilbert} is a notable exception, where a convergence rate of $\mathfrak{r}_n = n^{1/2} (\log n)^{-2}$ is obtained if the sequence $\{X_k\}_{k \in \Z}$ is geometrically $\varphi(n)$-mixing and satisfies some additional regularity assumptions (cf. Section \ref{sec_main_non_lin}).\\
\\
The aim of this note is twofold. In case of linear processes, we first give a Berry-Essen result with optimal rate under sharp moment assumptions ($p \in (2,3]$) and dependence conditions. We also show that the convergence rate may be arbitrarily slow. For non-linear processes, we first study one-dependent Bernoulli-shift sequences and establish the optimal rate. Based on this result, we then consider $m$-dependent Bernoulli-shift sequences and causal, weakly dependent Bernoulli-shift sequences with geometric decay in the dependence. In the latter, we obtain a convergence rate of $(n/\log n)^{1/2}$. For $m$-dependent processes, we obtain the optimal rate $(n/m)^{1/2}$.\\
\\
This note is structured as follows. In Sections \ref{sec_main} and \ref{sec_main_non_lin} the main results are presented and discussed. Proofs are given in Section \ref{sec_proofs}. Throughout the remainder, we make the following convention. For $p \geq 1$, denote with $\|\cdot\|_p$ the $L^p$-norm $\E[|\cdot|^p]^{1/p}$. We write $\lesssim$, $\gtrsim$, ($\thicksim$) to denote (two-sided) inequalities involving a multiplicative constant. Given a set $\mathcal{A}$, we denote with $\mathcal{A}^c$ its complement.

\section{Main results: Linear Processes}\label{sec_main}

Let us first introduce some additional necessary notation. Denote with $\bigl\{\xi_k\bigr\}_{k \in \Z} \in \HH$ an IID sequence of centered Gaussian random variables, where we require that the covariance operators of $\epsilon_k$ and $\xi_k$ are equal. We then consider the Gaussian counter part of $X_k$, namely
\begin{align*}
Z_k = \sum_{j \in \Z} \alpha_j(\xi_{k + j}), \quad k \in \Z.
\end{align*}
For $k \in \Z$, we also introduce the following (linear) operators, mapping from $\HH$ to $\HH$.
\begin{align}\nonumber \label{defn_cov_operator_lin}
{\mathbf A} &= \sum_{j \in \Z} \alpha_j, \quad {\mathbf A}_{n,k} = \sum_{j = -n + k}^{k-1} \alpha_j, \\
{\bf A}_{n,k}^c &= \sum_{j < -n+k} \alpha_j + \sum_{j > k-1} \alpha_j,\quad {\mathbf \Lambda} = {\mathbf A} \mathbf{C}^{\epsilon} {\mathbf A}^*,
\end{align}
where ${\bf B}^*$ denotes the adjoint of an operator ${\bf B}$. Note that we assign ${\mathbf \Lambda}$ a more concrete form here, and indeed one readily verifies that for $x \in \HH$
\begin{align*}
{\mathbf A} \mathbf{C}^{\epsilon} {\mathbf A}^*(x) = \E\bigl[\langle Z_{\bf \Lambda}, x \rangle Z_{\bf \Lambda} \bigr], \quad \text{where $Z_{\bf \Lambda} = {\bf A}(\xi_0)$.}
\end{align*}
One of the fundamental tools when working with linear processes is the elementary and well-known Beveridge and Nelson decomposition (BND) (cf. ~\cite{beveridge_nelson_1981})
\begin{align*}
\sum_{k = 1}^n X_k &= \sum_{k = 1}^n \sum_{j = -n + k}^{k-1} \alpha_j(\epsilon_k) + \sum_{k > n} \sum_{j = - n + k}^{k - 1} \alpha_j(\epsilon_k) + \sum_{k < 1} \sum_{j = -n + k}^{k - 1}\alpha_j(\epsilon_{k})\\&= \sum_{k = 1}^n {\bf A}(\epsilon_k) - \sum_{k = 1}^n {\bf A}_{n,k}^c(\epsilon_k) + \sum_{k > n} {\mathbf A}_{n,k}(\epsilon_k) + \sum_{k < 1} {\mathbf A}_{n,k}(\epsilon_{k}).
\end{align*}
It should be mentioned though that related, much more general martingale decompositions have already appeared earlier in the literature, see for instance ~\cite{gordin_1969} and ~\cite{hannan_1973}. For the CLT, $S_n(\epsilon) = \sum_{k = 1}^n {\mathbf A}\bigl(\epsilon_k\bigr)$ is the relevant part in \eqref{eq_decomp}, and indeed we have that
\begin{align*}
n^{-1/2}S_n(\epsilon) \xrightarrow{w} Z_{\bf \Lambda} \quad \text{if $\sum_{j \in \Z}\|\alpha_j\|_{\HHS} < \infty$,}
\end{align*}
see for instance ~\cite{merlevede_1997_opt}. Unlike to the real-valued case, condition $\sum_{j \in \Z}\|\alpha_j\|_{\HHS} < \infty$ is sharp in the sense that if it fails, no CLT can hold, even not under a possibly different normalization, see ~\cite{merlevede_1997_opt}. The corresponding counter example itself is set in the Gaussian domain, i.e. $\epsilon_k = \xi_k$, and solely relies on properties of the constructed sequence of linear operators $\alpha_{j}$. Thus, to a good proportion, the question of Berry-Esseen type results for linear processes is intimately connected to distributional properties of Gaussian random variables in Hilbert spaces.

Here we use results from ~\cite{ulyanov_1986} (cf. Lemma \ref{lem_ulyanov}), and particularly Lemma \ref{lem_gauss_hoelder_inf} as our main tools for the linear case. This requires us to impose some conditions on the eigenvalues of ${\bf \Lambda}$, which we denote with $\bigl\{\lambda_j\bigr\}_{j \in \N}$. We then derive our main results under the following assumptions.

\begin{ass}\label{ass_main}
For some $2 < p \leq 3$ it holds that
\begin{description}
\item[(i)] $\E\bigl[\epsilon_k\bigr] = 0$ and $\E\bigl[\|\epsilon_k\|_{\HH}^p\bigr]< \infty$,
\item[(ii)] $\sum_{j \in \Z}\|\alpha_j\|_{\HHS} < \infty$,
\item[(iii)] $\min_{1 \leq j \leq 13}\lambda_j > 0$ for ${\bf \Lambda}$ defined in \eqref{defn_cov_operator_lin}.
\end{description}
\end{ass}

Note that Assumption \ref{ass_main} (ii) implies that ${\mathbf A}_{n,k}, {\mathbf A}$ and ${\mathbf \Lambda}$ all exist and are of trace class. Our main result of this section is given below.

\begin{thm}\label{thm_berry_esseen_II}
Grant Assumption \ref{ass_main} and let $\mu \in \HH$ with $\|\mu\|_{\HH} < \infty$. Then
\begin{align*}
\Delta_n(\mu) \lesssim n^{-\frac{p}{2} + 1}\bigl(1 + \|\mu\|_{\HH}^{p}\bigr)\E\bigl[\|\epsilon_0\|_{\HH}^p\bigr] + n^{-1}\sum_{j \in \Z}(|j|\wedge n) \bigl\|\alpha_j \bigr\|_{\HHS}\E\bigl[\|\epsilon_0\|_{\HH}^p\bigr].
\end{align*}
The constant in $\lesssim$ only depends on $\sum_{j \in \Z}\|\alpha_j\|_{\HHS}$ and $\min_{1 \leq j \leq 13}\lambda_j$.
\end{thm}

\begin{rem}\label{rem_ass_lambda}
The primary objective of Theorem \ref{thm_berry_esseen_II} is to provide tight bounds in terms of the rate, the sequence $\{\alpha_j\}_{j \in \N}$ and the underlying moments $p \in (2,3]$. Note however if $\mu = 0$ and $p > 3$ the results in ~\cite{bentkus_goetze_1996} suggest that the rate can be improved. Observe also if Assumption \ref{ass_main} (iii) is violated (or in fact if $\lambda_j = 0$ for some finite $j$), we are facing a multivariate problem, which has been the subject of intensive study (cf. ~\cite{goetze_hipp_1983},~\cite{sweeting}). Again results dealing with independent random variables suggest that Assumption \ref{ass_main} (iii) may be weakened, see e.g. ~\cite{zalesski_1988} and ~\cite{prokhorov_ulyanov_2013} for a general overview.
\end{rem}

Unlike to the IID case, the rate of convergence is also governed by the additional component
\begin{align*}
\mathfrak{A}_n = n^{-1}\sum_{j \in \Z}(|j|\wedge n) \bigl\|\alpha_j \bigr\|_{\HHS}.
\end{align*}
Before discussing the bound $\mathfrak{A}_n$ in more detail, we state optimality of the above result.
\begin{thm}\label{thm_lower_bound}
Grant Assumption \ref{ass_main} and let $\mu \in \HH$ with $\|\mu\|_{\HH} < \infty$. Then the upper bound in Theorem \ref{thm_berry_esseen_II} is sharp up to a constant, that is, there exist examples meeting Assumption \ref{ass_main} where the upper bound is reached up to a constant.
\end{thm}

Expression $\mathfrak{A}_n$ is particularly interesting if it dominates the rate, that is, $n^{-\frac{p}{2} + 1} = \oo\bigl(\mathfrak{A}_n\bigr)$ and hence $\rd_n = \mathfrak{A}_n^{-1}$. In order to develop this a little further, let us consider real valued functions $f$ where
\begin{align}\label{eq_fn}
\text{$f(x) \geq 0$ is monotone and $\int_0^{\infty} f(x) dx < \infty$.}
\end{align}
Put $b_j = j^2 f(j)$ for $j \in \N$ and let $\alpha_j = \frac{b_j - b_{j-1}}{j}$ for $j \in \N\setminus\{0\}$ and $\alpha_j = 0$ otherwise. Then
\begin{align*}
\sum_{j \in \Z} (|j|\wedge n) |\alpha_j| \geq \sum_{j = 1}^n j \alpha_j = b_n = n^2 f(n).
\end{align*}
On the other hand, using the monotonicity of $f(n)$, we have $f(n) n \to 0$ as $n \to \infty$, and hence summation by parts yields
\begin{align*}
\sum_{j = 1}^{n} \alpha_j \leq \frac{b_n}{n} + \sum_{j \in \N} \frac{b_j}{j^2} \leq \oo\bigl(1\bigr) + \sum_{j \in \N} f(j) < \infty, \quad n \in \N.
\end{align*}
We thus obtain the following corollary.
\begin{cor}\label{cor_bad_rate}
Let $f(x)$ be a function satisfying \eqref{eq_fn}. Then there exist examples satisfying Assumption \ref{ass_main} where $\rd_n \thicksim (n f(n))^{-1}$.
\end{cor}
Corollary \ref{cor_bad_rate} gives a very simple method to provide upper bounds for the rate $\rd_n$. For example, setting
\begin{align*}
f(x) = (1 + x)^{-\ad} \quad \text{for $\ad > 1$ and $x \geq 0$ and $f(x) = 1$ if $x < 0$}
\end{align*}
gives the upper bound $n^{-\ad + 1}$. Logarithmic rates are obtained by $f(x) = (1 + x)^{-1} \bigl(\log(1 + x)\bigr)^{-\ad}$, $\ad > 1$, and this can be continued in the obvious way. Let us mention here that for the real valued case it is shown already in ~\cite{el_machkouri_volny_2004} that the rate of convergence in the CLT can be arbitrarily slow, where a much more general framework is considered. Let us now address the question when the rate $\rd_n = n^{\frac{p}{2} - 1}$ persists. To this end, put $\beta = \frac{p}{2} - 1$. Since we have the bound
\begin{align*}
\sum_{j \in \Z}(|j|\wedge n) \bigl\|\alpha_j \bigr\|_{\HHS} \leq n^{1 - \beta} \sum_{j \in \Z} |j|^{\beta} \bigl\|\alpha_j\bigr\|_{\HHS},
\end{align*}
we obtain the following corollary.

\begin{cor}\label{cor_fast}
Grant Assumption \ref{ass_main}. If we have in addition
\begin{align*}
\sum_{j \in \Z} |j|^{\frac{p}{2}-1}\bigl\|\alpha_j\bigr\|_{\HHS} < \infty, \quad p \in (2,3],
\end{align*}
then
\begin{align*}
\Delta_n(\mu) \lesssim n^{-\frac{p}{2} + 1}\bigl(1 + \|\mu\|_{\HH}^{p}\bigr)\E\bigl[\|\epsilon_0\|_{\HH}^p\bigr].
\end{align*}
\end{cor}

\section{Main results: Non-Linear Processes}\label{sec_main_non_lin}

As mentioned earlier, it appears that the only result which obtains optimal rates up to logarithmic factors is ~\cite{tikhomirov_1991_BE_Hilbert}, where $\{X_k\}_{k \in \Z}$ is required to be geometrically $\varphi(n)$-mixing, $X_i, X_j$ are uncorrelated for $i \neq j$ and
\begin{align*}
\E\bigl[|\langle X_k, h \rangle|^3\bigr] \lesssim \|h\|_{\HH}\E\bigl[|\langle X_k, h \rangle|^2\bigr], \quad h \in \HH, \, k \in \Z.
\end{align*}
In this Section, we follow a different path and focus on Bernoulli-shift processes. We first consider the special case of one-dependent sequences. To this end, let $\{\epsilon_k\}_{k \in \Z}$ be a sequence of IID random variables in some measure space $\mathbb{S}$, and $g:\mathbb{S} \mapsto \HH$ be a measurable map such that
\begin{align}\label{rep_two_dependence}
X_k = g(\epsilon_k, \epsilon_{k-1}), \quad k \in \Z.
\end{align}
Regarding the method of proof, this special structure will allow us to redirect the problem to the independent case (subject to a special conditional probability measure), by employing a conditioning argument. Unfortunately, as the proof shows, setting this idea to work leads to some non-trivial technicalities that need to be dealt with. To overcome these obstacles, we need to impose slightly stronger moment assumptions on $X_k$ than before.\\
\\
For our main result, Theorem \ref{thm_non_lin} below, we do not need to impose any additional conditions on $\mathbb{S}$, allowing for a large flexibility. This is demonstrated for instance by the subsequent Corollaries \ref{cor_m_dependence} and \ref{cor_geo_dependence}, where more general processes are considered. Our main assumptions are now the following.

\begin{ass}\label{ass_main_non_lin}
For some $ p \geq 9/2$ it holds that
\begin{description}
\item[(i)] $\E\bigl[\epsilon_k\bigr] = 0$ and $\E\bigl[\|X_k\|_{\HH}^p\bigr]< \infty$,
\item[(ii)] $\{X_k\}_{k \in \Z}$ satisfies \eqref{rep_two_dependence},
\item[(iii)] $\min_{1 \leq j \leq 13}\lambda_j > 0$ for ${\bf \Lambda}(\cdot)  = \sum_{|k| \leq 1}\E\bigl[\langle X_k, \cdot \rangle X_0\bigr]$.
\end{description}
\end{ass}

We then have the following result.

\begin{thm}\label{thm_non_lin}
Grant Assumption \ref{ass_main_non_lin}. Then
\begin{align*}
\Delta_n(\mu) \lesssim n^{-1/2}\bigl(1 + \|\mu\|_{\HH}^3\bigr)\E\bigl[\|X_0\|_{\HH}^{9/2}\bigr].
\end{align*}
The constant in $\lesssim$ only depends on $\min_{1 \leq j \leq 13}\lambda_j$.
\end{thm}

Compared to the linear case, the moment condition $p \geq 9/2$ appears to be suboptimal. On the other hand, for $\mu \neq 0$, the rate $n^{1/2}$ is optimal also for $p > 3$, see for instance ~\cite{bentkus_goetze_1996}.\\
\\
The flexibility in the setup allows us to treat Hilbert space valued $m$-dependent \textit{potential functions} (cf. ~\cite{goetze_Hipp_1989_filed_ptrf} for the real valued analogue). More precisely, for $m \in \N$, let
\begin{align}\label{potential_functions}
X_k = g_m(\epsilon_k, \epsilon_{k-1}, \ldots, \epsilon_{k-m+1}), \quad k \in \Z,
\end{align}
for measurable functions $g_m:\mathbb{S}^{m} \mapsto \HH$. We explicitly allow that $m = m_n$ with $m = \oo(n)$ may depend on the sample size $n$. The crucial condition here is the non-degeneracy assumption
\begin{align}\label{eq_var_condi}
\liminf_{n \to \infty} \E\bigl[\|S_n(X)\|_{\HH}^2 \bigr]/(nm) > 0.
\end{align}
The underlying covariance operator is then given as
\begin{align}\label{defn_cov_m_dpendent}
{\bf \Lambda}_{m}(\cdot) = m^{-2}\sum_{|l| \leq 1}\E\bigl[\langle B_l, \cdot \rangle B_0\bigr], \quad B_l = \sum_{k = (l-1)m +1}^{lm} X_k.
\end{align}
We now modify Assumption \ref{ass_main_non_lin} (iii) to
\begin{align}\label{eigen_m_dpendent}
\inf_{m} \min_{1 \leq j \leq 13}\lambda_{j,m} > 0, \quad \text{with ${\bf \Lambda}_{m}(\cdot)$ as in \eqref{defn_cov_m_dpendent},}
\end{align}
to obtain the following corollary.
\begin{cor}\label{cor_m_dependence}
Grant Assumption \ref{ass_main_non_lin} (i), and assume in addition the validity of \eqref{potential_functions}, \eqref{eq_var_condi} and \eqref{eigen_m_dpendent} with $m = \oo(n)$. Then
\begin{align*}
&\sup_{x \in \R}\Big|P\Bigl(\Bigl\|(nm)^{-1/2}S_n(X) + \mu\Bigr\|_{\HH} \leq x \Bigr) - P\Bigl(\Bigl\|Z_{{\bf \Lambda}_{m}} + \mu\Bigr\|_{\HH} \leq x \Bigr) \Bigr| \\&\lesssim (n/m)^{-1/2} \bigl(1 + \|\mu\|_{\HH}^{3}\bigr)\E\bigl[\|X_0\|_{\HH}^{9/2}\bigr].
\end{align*}
\end{cor}

Recall that the rate $(n/m)^{1/2}$ is optimal even for real-valued cases, see for instance ~\cite{chen_2004_aop},~\cite{tikhomirov_1981} for analogue univariate and multivariate results (a reparametrization is necessary to obtain this explicit form of the rate), and ~\cite{sunklodas_1997} for a lower bound.

A different dependence setup is if $\{X_k\}_{k \in \Z}$ exhibits weak dependence, the latter only coinciding with $m$-dependency in general if $m$ is finite and independent of $n$. A huge variety of weak dependence concepts have been discussed in the literature, see for example ~\cite{dedecker_prieur_2005} and ~\cite{wu_2005}. In our context, the notion of Bernoulli-shift processes together with coupling coefficients is particularly useful (cf. ~\cite{wu_2005}). For Hilbert space valued processes, a related concept is $\Ln^p-m$ approximability, see ~\cite{hoermann_2010}. To formalise the setup, consider
\begin{align}\label{bernoulli_functions}
X_k = g(\epsilon_k, \epsilon_{k-1}, \ldots ) \quad k \in \Z,
\end{align}
for measurable functions $g:\mathbb{S}^{\infty} \mapsto \HH$. Let $\{\epsilon_k'\}_{k \in \Z}$ be an independent copy of $\{\epsilon_k\}_{k \in \Z}$. We then define the 'coupled' random variable $X_k'$ as
\begin{align*}
X_k' = g(\epsilon_k,\ldots,\epsilon_1,\epsilon_0',\epsilon_{-1},\ldots) \quad k \in \N,
\end{align*}
see ~\cite{wu_2005} for more details on this kind of coupling. Dependence measures can now be constructed by measuring the distance between $X_k$ and $X_k'$, a popular measure being
\begin{align}\label{bernoulli_coupling}
\theta_p(k) = \E\bigl[\|X_k-X_k'\|_{\HH}^p\bigr]^{1/p}, \quad p \geq 1,
\end{align}
which we use in the sequel. In the presence of infinite dependence, the underlying covariance operator is now (formally) defined as
\begin{align}\label{bernoulli_cov}
{\bf \Lambda}(\cdot) = \sum_{k \in \Z} \E\bigl[\langle X_k, \cdot \rangle X_0 \bigr].
\end{align}
Existence holds if $\sum_{k \in \N} \theta_2(k) < \infty$, see for instance ~\cite{dedecker_2003}. As before, we modify Assumption \ref{ass_main_non_lin} (iii) to
\begin{align}\label{eigen_bernoulli}
\min_{1 \leq j \leq 13}\lambda_{j} > 0, \quad \text{with ${\bf \Lambda}_{}(\cdot)$ as in \eqref{bernoulli_cov}.}
\end{align}
We then have the following result.
\begin{cor}\label{cor_geo_dependence}
Grant Assumption \ref{ass_main_non_lin} (i), and assume the validity of \eqref{bernoulli_functions} and \eqref{eigen_bernoulli}.  If in addition $\theta_{9/2}(k) \lesssim \rho^k$, $0 < \rho < 1$, then
\begin{align*}
\Delta_n(\mu) \lesssim (n/\log n)^{-1/2} \bigl(1 + \|\mu\|_{\HH}^{3}\bigr)\E\bigl[\|X_0\|_{\HH}^{9/2}\bigr].
\end{align*}
\end{cor}

The literature provides a huge variety of important examples of processes displaying a geometric decay in $\theta_{p}(k)$. A prominent example is the following.

\begin{ex}[ARCH-processes]
Let $\beta \in \Ln^2\bigl([0,1]^2\bigr)$ be a non-negative Kernel and $\{\epsilon_{k}\}_{k \in \Z} \in \Ln^2([0,1])$ be an IID sequence with $\E[\epsilon_k] = 0$. If the function $\mu \in \Ln^2([0,1])$ is positive, then we call the process
\begin{align*}
X_k = X_k(t) = \epsilon_k(t) \sigma_{k-1}(t), \quad k \in \Z, \, t \in [0,1],
\end{align*}
with
\begin{align*}
\sigma_k^2(t) = \mu(t) + \int_0^1 \beta(t,s) X_{k-1}^2(s)ds, \quad k \in \Z, \, t \in [0,1],
\end{align*}
the functional ARCH(1)-process. To see why $X_k$ fits into our framework (satisfies representation \eqref{bernoulli_functions}) is by formally iterating the recursion, yielding (with $t = t_0$)
\begin{align*}
X_k(t_0) = \epsilon_k(t_0) \Big(\mu(t_0) + \sum_{i= 1}^{\infty} \prod_{j = 1}^{i}\int_0^1 \beta(t_{j-1},t_j) \epsilon_{k-j}^2(t_j) \tilde{\mu}_i(j,t_j) d t_j\Big)^{1/2},
\end{align*}
where $\tilde{\mu}_i(j,\cdot) = 1 $ for $i \neq j$ and $\tilde{\mu}_i(i,\cdot) = \mu(\cdot)$.
\begin{comment}
\begin{align*}
\sigma_k^2(t_0) = \mu(t_0) + \sum_{i= 1}^{\infty} \prod_{j = 1}^{i}\int_0^1 \beta(t_{j-1},t_j) \epsilon_{k-j}^2(t_j) \mu(t_j) d t_j.
\end{align*}
\end{comment}
This formal argument can be made rigorous by using Proposition 2.3 in ~\cite{hoermann_2010} (see also ~\cite{hoermann_horvath_reeder_2013}, Theorem 2.1), provided that
$\E\bigl[K^p(\epsilon_0^2)\bigr] < 1$ with $p \geq 2$, where
\begin{align*}
K^2(\epsilon_0^2) = \int_0^1\int_0^1 \beta^2(s,t) \epsilon_0^4(s) ds.
\end{align*}
Moreover, Proposition 2.3 in ~\cite{hoermann_2010} also implies $\theta_{p}(k) \lesssim \rho^{k}$ with $0 < \rho < 1$ for $p \geq 2$. Hence if $p \geq 9/2$ and \eqref{eigen_bernoulli} holds, Corollary \ref{cor_geo_dependence} applies.
\end{ex}

For additional examples with geometric decay, we refer to ~\cite{hoermann_2010}.

\section{Proofs}\label{sec_proofs}

We first deal with the results concerning linear processes, given in Section \ref{sec_main}. This is then followed by the proofs of Section \ref{sec_main_non_lin}. We first collect and review some required results from the literature we make repeated use of.\\
\\
Consider two compact operators $\mathbf{K}, \mathbf{L}$ with singular value decompositions
\begin{align}\label{eq_singular_deco}
\mathbf{K}(x) = \sum_{j \in \N} \lambda_j^K \langle x, u_j \rangle f_j, \quad \mathbf{L}(x) = \sum_{j \in \N} \lambda_j^L \langle x, v_j \rangle g_j.
\end{align}
The following lemma is proven in Section VI.1 of Gohberg et al. ~\cite{gohberg_2003}, see their Corollary 1.6 on p. 99

\begin{lem}\label{lem_eigen_gen_upper_bound}
Let $\mathbf{K}$ and $\mathbf{L}$ be compact operators with singular value decompositions as in \eqref{eq_singular_deco}. Then
\begin{align*}
\bigl|\lambda_j^L - \lambda_j^K\bigr| \leq \bigl\|{\mathbf K} - \mathbf{L}\bigr\|_{\HHS}.
\end{align*}
\end{lem}

The next lemma appears in some variants in the literature, see for instance ~\cite{bentkus_1984}, ~\cite{ulyanov_1986} and ~\cite{yurinskii_1982}.

\begin{lem}\label{lem_gauss_char}
Let $Z \in \HH$ be a zero mean Gaussian random variable and ${\bf C}^{Z}$ its covariance operator. Let $Y \in \HH$ be another, independent random variable. Then
\begin{align*}
\bigl|\E\bigl[\exp(\ic t \|Z + Y\|_{\HH}^2)\bigr] \bigr| \leq \prod_{k = 1}^{\infty} \bigl(1 + 4 t^2 (\lambda_k^Z)^2 \bigr)^{-1/4},
\end{align*}
where $\lambda_k^Z$ denotes the eigenvalues of ${\bf C}^Z$.
\end{lem}

For the next lemma, we assume that $\bigl\{Y_j\bigr\}_{j \in \N} \in \HH$ is an independent sequence. For $2 \leq p \leq 3$, we introduce the quantities
\begin{align*}
\mathfrak{M}_{p,j} &= \E\bigl[\bigl\|Y_j \ind_{\|Y_j\|_{\HH} > 1}\bigr\|_{\HH}^p\bigr], \quad \mathfrak{M}_p = \sum_{j \in \N} \mathfrak{M}_{p,j},\\
\mathfrak{L}_{p,j} &= \E\bigl[\bigl\|Y_j \ind_{\|Y_j\|_{\HH} \leq 1}\bigr\|_{\HH}^p\bigr], \quad \mathfrak{L}_p = \sum_{j \in \N} \mathfrak{L}_{p,j}.
\end{align*}
Let ${\bf \Lambda}_Y$ be the covariance operator of $S\bigl(Y\bigr) = \sum_{j \in \N} Y_j$. The following result is an adapted version of Theorem 1 in ~\cite{ulyanov_1986}.

\begin{lem}\label{lem_ulyanov}
Assume that the first thirteen eigenvalues of ${\bf \Lambda}_Y$ are strictly positive. Then for any $2 < p \leq 3$ and $a \in \HH$, we have
\begin{align*}
\sup_{x \in \R}\biggl|P\biggl(\bigr\|S(Y) - a\bigr\|_{\HH} \leq x \biggr) - P\biggl(\bigr\|Z - a\bigr\|_{\HH} \leq x \biggr) \biggr| \lesssim \bigl(1 + \bigl\|a\bigr\|_{\HH}^p\bigr)\bigl(\mathfrak{M}_2 + \mathfrak{L}_p\bigr),
\end{align*}
where $Z \in \HH$ is a Gaussian random variable with covariance operator ${\bf \Lambda}_Y$.
\end{lem}

\subsection{Proofs of Section \ref{sec_main}}\label{sec_proof_lin}

We first state and prove the following auxiliary result.

\begin{thm}\label{thm_berry_esseen}
Grant Assumption \ref{ass_main} and let $\mu \in \HH$ with $\|\mu\|_{\HH} < \infty$. Then
\begin{align*}
&\sup_{x \in \R}\bigl|P\bigl(\|n^{-1/2}S_n(X) + \mu\|_{\HH} \leq x \bigr) - P\bigl(\|Z_{\bf \Lambda_n} + \mu\|_{\HH} \leq x \bigr) \bigr| \lesssim n^{-\frac{p}{2} + 1} \bigl(1 + \|\mu\|_{\HH}^{p}\bigr)\E\bigl[\|\epsilon_0\|_{\HH}^p\bigr],
\end{align*}
for an appropriate covariance operator ${\bf \Lambda}_n(\cdot)$. The constant in $\lesssim$ only depends on $\sum_{j \in \Z}\|\alpha_j\|_{\HHS}$ and $\min_{1 \leq j \leq 13}\lambda_j$.
\end{thm}

Theorem \ref{thm_berry_esseen} gives the optimal rates under sharp dependence condition Assumption \ref{ass_main} (ii). Note that here the underlying covariance operator ${\bf \Lambda}_n$ depends on $n$ (see the proof for the precise construction of ${\bf \Lambda}_n$). Based on this result, we then obtain Theorem \ref{thm_berry_esseen_II} based on the comparison Lemma \ref{lem_gauss_hoelder_inf} for Hilbert space valued Gaussian random variables.

\begin{proof}[Proof of Theorem \ref{thm_berry_esseen}]
For $1 \leq k \leq n$, put $\sqrt{n}U_k = {\mathbf A}_{n,k}(\epsilon_k)$ and $\sqrt{n}U_k = {\mathbf A}_{n,k}(\epsilon_k) + {\mathbf A}_{n,-k+n}(\epsilon_{-k+n})$ for $n + 1 \leq k \leq \infty$. We then use the abbreviations $T_k = U_k$ for $1 \leq k \leq n$ and $T_{n+1} = \sum_{k = n+1}^{\infty} U_k$. Moreover, we put $S_{1}^n(U) = \sum_{k = 1}^n U_k$ and $S_{1}^n(T) = \sum_{k = 1}^n T_k$. Then
\begin{align}\label{eq_decomp}
n^{-1/2}\sum_{k = 1}^n X_k = S_1^{n+1}(T).
\end{align}
In addition, we denote with $S_1^n(Z)$ the Gaussian counter parts, that is, every $\epsilon_i$ is replaced with $\xi_i$ at the corresponding places. For $x \in \HH$ denote with
\begin{align*}
{\mathbf \Lambda}_{n}^{(0)}(x) &= n^{-1} \sum_{k = 1}^n \E\bigl[\langle X_k,x \rangle X_k \bigr],\\
{\mathbf \Lambda}_{n,k}(x) &= {\mathbf A}_{n,k} \mathbf{C}^{\epsilon} {\mathbf A}_{n,k}^*(x),\\
{\mathbf \Lambda}_{n}^{(1)}(x) &= n^{-1} \sum_{k = 1}^n {\mathbf A}_{n,k} \mathbf{C}^{\epsilon} {\mathbf A}_{n,k}^*,
\end{align*}
and with $\bigl\{\lambda_{n,j}\bigr\}_{j \in \N}$ the eigenvalues of the covariance operator ${\mathbf \Lambda}_{n}^{(0)}$. The proof of Theorem \ref{thm_berry_esseen} requires the following lemmas.

\begin{lem}\label{lem_control_Lambda_n+1}
Grant Assumption \ref{ass_main}. Then for any $2 \leq p \leq 3$ we have
\begin{align*}
&\E\bigl[\|T_{n+1}\|_{\HH}^p\bigr] = \oo\bigl(n^{-\frac{p}{2} + 1}\E\bigl[\|\epsilon_0\|_{\HH}^p\bigr] \bigr) \quad \text{and}\\
&\bigl\|{\bf \Lambda}_n^{(0)} - {\bf \Lambda}_{n}^{(1)} \bigr\|_{\HHS} = \oo\bigl(1\bigr).
\end{align*}
\end{lem}

\begin{proof}[Proof of Lemma \ref{lem_control_Lambda_n+1}]
Due to the triangle inequality, it follows that for $n + 1 \leq i \leq \infty$ and $p \geq 1$ we have
\begin{align}\nonumber \label{eq_ui_bound_prelim_0}
\E\bigl[\bigl\|\sqrt{n}U_i\bigr\|_{\HH}^p\bigr] &\leq  \Bigl(\sum_{j = 1 - i}^{n-i} \bigl\|\alpha_j\bigr\|_{\HHS} + \sum_{j = 1 + i}^{n+i} \bigl\|\alpha_j\bigr\|_{\HHS}\Bigr)^p\E\bigl[\|\epsilon_0\|_{\HH}^p\bigr] \\&\lesssim \Bigl(\sum_{j = 1 + i}^{n+i} \bigl\|\alpha_{-j}\bigr\|_{\HHS} + \sum_{j = 1 + i}^{n+i} \bigl\|\alpha_j\bigr\|_{\HHS}\Bigr)^p \E\bigl[\|\epsilon_0\|_{\HH}^p\bigr].
\end{align}
Denote with $a_{n,i}^+ = \sum_{j = 1 + i}^{n+i}\bigl\|\alpha_j\bigr\|_{\HHS}$ and $a_{n,i}^- = \sum_{j = 1 + i}^{n+i}\bigl\|\alpha_{-j}\bigr\|_{\HHS}$. Note that since $\sum_{j \in \Z}\bigl\|\alpha_j\bigr\|_{\HHS} < \infty$ we have
\begin{align}\label{eq_ai_bound_prelim_1}
a_{n,i}^+, a_{n,i}^- \to 0 \quad \text{as $i \to \infty$, uniformly in $n$.}
\end{align}
Then by \eqref{eq_ui_bound_prelim_0}, it follows that for $K \in \N$
\begin{align*}
\sum_{i = n+1}^{\infty} \E\bigl[\|\sqrt{n}U_i\|_{\HH}^p\bigr]/\E\bigl[\|\epsilon_0\|_{\HH}^p\bigr] &\lesssim \sum_{i = 1}^{K}(a_{n,i}^+ + a_{n,i}^-)^p + \sum_{i > K} (a_{n,i}^+ + a_{n,i}^-)^p \\&\lesssim 2^p K \Bigl(\sum_{i \in \Z}\|\alpha_i\|_{\HHS} \Bigr)^p + \sum_{|i| > K} (n \wedge i) \|\alpha_i\|_{\HHS}.
\end{align*}
Selecting $K = K_n \to \infty$ such that $K_n = \oo\bigl(n)$, we deduce that
\begin{align}\label{eq_ui_bound_prelim}
\sum_{i = n+1}^{\infty} \E\bigl[\|\sqrt{n}U_i\|_{\HH}^p\bigr] = \oo\bigl(n \E\bigl[\|\epsilon_0\|_{\HH}^p\bigr]\bigr).
\end{align}
Using a Rosenthal inequality for Hilbert spaces (cf. ~\cite{osekowski_2012}), we get for $p  \geq 2$
\begin{align}\label{eq_vn_bound_prelim_1} \nonumber
\E\bigl[\|T_{n+1}\|_{\HH}^p\bigr] &\lesssim n^{-\frac{p}{2}} \left\{\Bigl(\sum_{i = n+1}^{\infty} \E\bigl[\|\sqrt{n}U_i\|_{\HH}^2\bigr] \Bigr)^{\frac{p}{2}} + \sum_{i = n+1}^{\infty} \E\bigl[\|\sqrt{n}U_i\|_{\HH}^p\bigr] \right\} \\&= \oo\bigl(n^{-\frac{p}{2} + 1}\E\bigl[\|\epsilon_0\|_{\HH}^p\bigr] \bigr).
\end{align}
This gives the first claim. Next, denote with $\lambda_{j}^{(T)}$ and $e_j^{(T)}$ the eigenvalues and functions of the Covariance operator of $T_{n+1}$, which exists due to \eqref{eq_vn_bound_prelim_1}. Since $T_{n+1}$ is independent of $\bigl\{T_j\bigr\}_{1 \leq j \leq n}$ by construction, we get that for any $x \in \HH$
\begin{align*}
{\bf \Lambda}_n^{(0)}(x) - {\bf \Lambda}_{n}^{(1)}(x) &= \E\bigl[\langle S_1^n(T), x\rangle T_{n+1} + \langle T_{n+1}, x\rangle S_{1}^n(T) + \langle T_{n+1}, x\rangle T_{n+1} \bigr] \\& = \E\bigl[\langle T_{n+1}, x\rangle T_{n+1} \bigr].
\end{align*}
It then follows from Cauchy-Schwarz and Parseval's identity that
\begin{align*}
\bigl\|{\bf \Lambda}_n^{(0)} - {\bf \Lambda}_{n}^{(1)} \bigr\|_{\HHS} \leq \bigl\|\E\bigl[\langle T_{n+1},\cdot \rangle T_{n+1} \bigr]\bigr\|_{\HHS} \leq \Big(\sum_{j = 1}^{\infty} \bigl(\lambda_{j}^{(T)}\bigr)^2\Big)^{1/2}.
\end{align*}
By the triangle inequality and from \eqref{eq_vn_bound_prelim_1}, we deduce that
\begin{align*}
\Big(\sum_{j = 1}^{\infty} \bigl(\lambda_{j}^{(T)}\bigr)^2\Big)^{1/2} \leq \sum_{j = 1}^{\infty}\lambda_{j}^{(T)} = \E\bigl[\|T_{n+1}\|_{\HH}^2\bigr] = \oo\bigl(1\bigr),
\end{align*}
and hence
\begin{align}
\bigl\|{\bf \Lambda}_n^{(0)} - {\bf \Lambda}_{n}^{(1)} \bigr\|_{\HHS} = \oo\bigl(1\bigr).
\end{align}
\end{proof}

\begin{lem}\label{lem_l_eigenvalue_positive}
Assume that Assumption \ref{ass_main} holds. Then there exists an $n_0 \in \N$ such that for $n \geq n_0$ we have $\lambda_{n,k} > 0$ for any fixed $k \in \N$ where $\lambda_k > 0$.
\end{lem}

\begin{proof}[Proof of Lemma \ref{lem_l_eigenvalue_positive}]
Note first that since $\bigl\|{\mathbf A}\bigr\|_{\HHS}, \bigl\|{\mathbf C}^{\epsilon}\bigr\|_{\HHS}, \bigl\|{\mathbf A}_{n,k}\bigr\|_{\HHS}<\infty$ for all $i \in \Z$ and $n \in \N$, ${\mathbf \Lambda}_{n,k}$ and ${\mathbf \Lambda}$ are bounded operators. By Lemma \ref{lem_eigen_gen_upper_bound} we have
\begin{align*}
\bigl|\lambda_{k} - \lambda_{n,k}\bigr| \leq \bigl\|{\mathbf \Lambda}-{\mathbf \Lambda}_{n}^{(1)}\bigr\|_{\HHS},
\end{align*}
hence it suffices to consider $\bigl\|{\mathbf \Lambda}-{\mathbf \Lambda}_{n}^{(1)}\bigr\|_{\HHS}$. Moreover, by the triangle inequality and Lemma \ref{lem_control_Lambda_n+1}, we only need to consider $\bigl\|{\mathbf \Lambda}-{\mathbf \Lambda}_{n}^{(1)}\bigr\|_{\HHS}$. Using the linearity of ${\mathbf A}_{n,k}, {\mathbf A},{\mathbf C}^{\epsilon}$ and the fact that $\bigl\|{\mathbf B}\bigr\|_{\HHS} = \bigl\|{\mathbf B}^*\bigr\|_{\HHS}$ for an operator ${\mathbf B}$, it follows that
%(A B)^* = B^* A *
\begin{align*}
\bigl\|{\mathbf \Lambda}_{n,k} - {\mathbf \Lambda}\bigr\|_{\HHS} &\leq \bigl\|{\mathbf A}_{n,k} \mathbf{C}^{\epsilon} {\mathbf A}_{n,k}^* - {\mathbf A}_{n,k} \mathbf{C}^{\epsilon} {\mathbf A}^*\bigr\|_{\HHS} + \bigl\|{\mathbf A} \mathbf{C}^{\epsilon} {\mathbf A}^* - {\mathbf A}_{n,k} \mathbf{C}^{\epsilon} {\mathbf A}^*\bigr\|_{\HHS} \\&\leq 2 \bigl\|{\mathbf A} - {\mathbf A}_{n,k}\bigr\|_{\HHS}\bigl\|{\mathbf C}^{\epsilon}\bigr\|_{\HHS} \bigl(\bigl\|{\mathbf A}\bigr\|_{\HHS} + \bigl\|{\mathbf A}_{n,k}\bigr\|_{\HHS} \bigr).
\end{align*}
By the triangle inequality, we have $\bigl\|{\mathbf A} - {\mathbf A}_{n,k}\bigr\|_{\HHS} \leq \sum_{j > n - k} \bigl\|\alpha_j\bigr\|_{\HHS} + \sum_{j < 1 - k} \bigl\|\alpha_j\bigr\|_{\HHS} = \oo\bigl(1\bigr)$ as $k,n \to \infty$. Since
\begin{align}
\bigl\|n^{-1}\sum_{k = 1}^{n}{\mathbf \Lambda}_{n,k} - {\mathbf \Lambda}\bigr\|_{\HHS} \leq n^{-1} \sum_{k = l}^{n-l}\bigl\|{\mathbf \Lambda}_{n,k} - {\mathbf \Lambda}\bigr\|_{\HHS} + 2l/n \bigl\|{\mathbf \Lambda}\bigr\|_{\HHS},
\end{align}
the claim follows.
\end{proof}

We are now ready to proceed to the proof of Theorem \ref{thm_berry_esseen}. Due to Lemma \ref{lem_l_eigenvalue_positive}, we have
\begin{align}\label{eq_eigen_13_positive}
\liminf_{n \to \infty} \min_{1 \leq k \leq 13} \lambda_{n,k} > 0.
\end{align}
Hence applying Lemma \ref{lem_ulyanov}, it follows that
\begin{align}\nonumber
\Delta_n(\mu) &\leq \sup_{x \in \R} \Bigl|P\bigl(\bigl\|S_1^{n+1}(T) + \mu\bigr\|_{\HH} \leq x \bigr) - P\bigl(\bigl\|S_1^{n+1}(Z) + \mu\bigr\|_{\HH} \leq x \bigr) \Bigr| \\& \nonumber \lesssim  \bigl(1 + \|\mu\|_{\HH}^p\bigr) \Bigl(\sum_{k = 1}^{n+1} \Bigl\|\|T_k\|_{\HH} \ind(\|T_k\|_{\HH} \geq 1)\Bigr\|_2^2 + \sum_{k = 1}^{n+1}\Bigl\|\|T_k\|_{\HH} \Bigr\|_p^p\Bigr)\\&\stackrel{def}{=} \bigl(1 + \|\mu\|_{\HH}^p\bigr)\bigl({\bf I}_n + {\bf II}_n\bigr).
\end{align}
We first treat ${\bf I}_n$. For $1 \leq k \leq n$, we obtain
\begin{align}\nonumber \label{eq_I_n_1}
\Bigl\|\|T_k\|_{\HH} \ind(\|T_k\|_{\HH} \geq 1)\Bigr\|_2^2 &= \bigl\|n^{-1/2}\|{\mathbf A}_{n,k}(\epsilon_k)\|_{\HH} \ind(\|{\mathbf A}_{n,k}(\epsilon_k)\|_{\HH} \geq n^{1/2})\bigr\|_2^2 \\&\nonumber\lesssim n^{-1 - \frac{p - 2}{2}} \bigl\|\|{\mathbf A}_{n,k}(\epsilon_k)\|_{\HH}\bigr\|_p^p\\&\lesssim n^{-\frac{p}{2}} \E\bigl[\|\epsilon_0\|_{\HH}^p\bigr].
\end{align}
Similarly, using Lemma \ref{lem_control_Lambda_n+1}, we deduce that %h = 1$
\begin{align}\label{eq_I_n_2}
\Bigl\|\|T_{n+1}\|_{\HH} \ind(\|T_{n+1}\|_{\HH} \geq 1)\Bigr\|_2^2 \lesssim n^{-\frac{p}{2} + 1}\E\bigl[\|\epsilon_0\|_{\HH}^p\bigr].
\end{align}
Combining \eqref{eq_I_n_1} and \eqref{eq_I_n_2}, we obtain
\begin{align}\label{eq_I_n_3}
{\bf I}_n \lesssim n^{-\frac{p}{2} + 1}\E\bigl[\|\epsilon_0\|_{\HH}^p\bigr].
\end{align}
Next, we deal with ${\bf II}_n$. First note that for $1 \leq k \leq n$, we obtain via the triangle inequality
\begin{align*}
\bigl\|\|T_k\|_{\HH} \bigr\|_p \lesssim n^{-\frac{1}{2}} \sum_{j \in \Z} \|\alpha_j\|_{\HHS} \bigl\|\|\epsilon_0\|_{\HH} \bigr\|_p \lesssim n^{-1/2} \E\bigl[\|\epsilon_0\|_{\HH}^p\bigr].
\end{align*}
Using Lemma \ref{lem_control_Lambda_n+1}, we thus deduce that
\begin{align}\label{eq_II_n_1}
{\bf II}_n \lesssim \sum_{j = 1}^{n+1} n^{-\frac{p}{2}}\E\bigl[\|\epsilon_0\|_{\HH}^p\bigr] \lesssim n^{-\frac{p}{2} + 1}\E\bigl[\|\epsilon_0\|_{\HH}^p\bigr].
\end{align}
Combining \eqref{eq_I_n_3} with \eqref{eq_II_n_1} completes the proof.
\end{proof}

%\subsubsection{Proof of Theorem \ref{thm_berry_esseen_II}}

\begin{proof}[Proof of Theorem \ref{thm_berry_esseen_II}]
The proof is based on two main lemmas. The first one describes a comparison result for two Gaussian, Hilbert-space valued random variables in terms of perturbed covariance operators.

\begin{lem}\label{lem_gauss_hoelder_inf}
Let $Y_1,Y_2 \in \HH$ be two Gaussian random variables with covariance operators ${\bf C}^{Y_1}$ and ${\bf C}^{Y_2}$ of trace class and finite mean $\mu$. Suppose that for $\delta > 0$
\begin{align}\label{eq_condi_lem_gauss_hoelder_inf}
\lambda_{13}^{Y_1} > \delta \quad \text{and} \quad \bigl\|{\bf C}^{Y_1} - {\bf C}^{Y_2} \bigr\|_{\HHS} \leq \delta,
\end{align}
where $\bigl\{\lambda_j^{Y_1}\bigr\}_{j \in \N}$ denotes the eigenvalues of ${\bf C}^{Y_1}$. Then
\begin{align*}
\sup_{x \in \R}\bigl|P\bigl(\|Y_1\|_{\HH} \leq x \bigr) - P\bigl(\|Y_2\|_{\HH} \leq x \bigr) \bigr| \lesssim \bigl|\tr\bigl({\bf C}^{Y_1}\bigr) - \tr\bigl({\bf C}^{Y_2}\bigr)\bigr|,
\end{align*}
where $\tr\bigl({\bf B}\bigr)$ denotes the trace of an operator ${\bf B}$.
\end{lem}

\begin{proof}[Proof of Lemma \ref{lem_gauss_hoelder_inf}]
We may argue similarly as in ~\cite{yurinskii_1982}. Due to the Gaussianity of $Y_1$ and $Y_2$
\begin{align}
Y_1 - \mu \stackrel{d}{=} \sum_{k = 1}^n \xi_{n,k}, \quad Y_2 - \mu \stackrel{d}{=} \sum_{k = 1}^n \eta_{n,k},
\end{align}
where $\bigl\{\xi_{n,k}\bigr\}_{1 \leq k \leq n}$ and $\bigl\{\eta_{n,k}\bigr\}_{1 \leq k \leq n}$ are IID Gaussian sequences with Covariance operators $n^{-1}{\bf C}^{Y_1}$ and $n^{-1}{\bf C}^{Y_2}$. For $n \in \N$ denote with
\begin{align}
W_{n,k} = \mu + \sum_{j = 1}^{k-1} \xi_{n,k} + \sum_{j = k + 1}^n \eta_{n,k}, \quad 1 \leq k \leq n.
\end{align}
Following the proof in ~\cite{yurinskii_1982}, a careful inspection reveals (cf. equation 3.5 in ~\cite{yurinskii_1982}) that it suffices to reconsider the quantity
\begin{align}
Q_{n,k}(t) = |t|\bigl|\E\bigl[\exp(\ic t \| W_{n,k} \|_{\HH}^2 \bigr]\bigl(\E[\|\xi_{n,k}\|_{\HH}^2] - \E[\|\eta_{n,k}\|_{\HH}^2]\bigr) \bigr|, \quad t \in \R,
\end{align}
and establish that
\begin{align}\label{eq_lem_gauss_hoelder_inf_3}
\int_{\R} \frac{Q_{n,k}(t) }{t }  d\,t \lesssim \frac{1}{n}\bigl|\tr\bigl({\bf C}^{Y_1}\bigr) - \tr\bigl({\bf C}^{Y_2}\bigr)\bigr|.
\end{align}
Once we have \eqref{eq_lem_gauss_hoelder_inf_3}, the results in ~\cite{yurinskii_1982} imply that
\begin{align*}
\sup_{x \in \R}\bigl|P\bigl(\|Y_1\|_{\HH} \leq x \bigr) - P\bigl(\|Y_2\|_{\HH} \leq x \bigr) \bigr| &\lesssim \frac{1}{n}\sum_{k = 1}^n\bigl|\tr\bigl({\bf C}^{Y_1}\bigr) - \tr\bigl({\bf C}^{Y_2}\bigr)\bigr| + \frac{1+\|\mu\|_{\HH}^3}{\sqrt{n}}\\& \lesssim \bigl|\tr\bigl({\bf C}^{Y_1}\bigr) - \tr\bigl({\bf C}^{Y_2}\bigr)\bigr| + \frac{1+\|\mu\|_{\HH}^3}{\sqrt{n}}.
\end{align*}
Selecting $n$ sufficiently large, the claim follows. We proceed by showing \eqref{eq_lem_gauss_hoelder_inf_3}. To this end, note that by Lemma \ref{lem_gauss_char} we have that
\begin{align}\label{eq_lem_gauss_hoelder_inf_4}
\bigl|\E\bigl[\exp(\ic t \| W_{n,k} \|_{\HH}^2 \bigr]\bigr| \leq \prod_{j = 1}^{\infty} \bigl(1 + 4 t^2 \tilde{\lambda}_{n,k,j}^2 \bigr)^{-1/4},
\end{align}
where $\tilde{\lambda}_{n,k,j}$ denote the eigenvalues of the covariance operator
\begin{align}
\tilde{\bf C}_{n,k}(x) = \E\bigl[\langle W_{n,k} - \mu, x \rangle(W_{n,k} - \mu) \bigr].
\end{align}
Exploiting the mutual independence of $\xi_{n,k}$ and $\eta_{n,k}$, it follows that
\begin{align*}
\bigl\|\tilde{\bf C}_{n,k} - {\bf C}^{Y_1} \bigr\|_{\HHS} &\leq \frac{1}{n}\bigl\|{\bf C}^{Y_1}\bigr\|_{\HHS} + \frac{n-k}{n}\bigl\|{\bf C}^{Y_2} - {\bf C}^{Y_1}\bigr\|_{\HHS} \\&\leq \delta + \OO\bigl(n^{-1}\bigr).
\end{align*}
Hence an application of Lemma \ref{lem_eigen_gen_upper_bound} yields that for $1 \leq k \leq 13$
\begin{align}\nonumber
\tilde{\lambda}_{n,k,j} &\geq \lambda_{k}^{Y_1} - \bigl|\lambda_{k}^{Y_1} - \tilde{\lambda}_{n,k,j}\bigr| \\&\geq \lambda_{k}^{Y_1} - \delta - \OO\bigl(n^{-1}\bigr) > 0
\end{align}
for sufficiently large $n$. Hence $\tilde{\lambda}_{n,k,j} > 0$ uniformly, and we conclude from \eqref{eq_lem_gauss_hoelder_inf_4} that
\begin{align}
\int_{\R} \bigl|\E\bigl[\exp(\ic t \| W_k \|_{\HH}^2 \bigr]\bigr| d \,t \leq \int_{\R} \prod_{j = 1}^{13} \bigl(1 + 4 t^2 \tilde{\lambda}_{n,k,j}^2 \bigr)^{-1/4} d\,t < \infty.
\end{align}
Note that here we actually only require that $\lambda_{3}^{Y_1} > 0$. Since we have that
\begin{align*}
\E\bigl[\|X\|_{\HH}^2\bigr] = \tr\bigl({\bf C}^X\bigr)
\end{align*}
for any $X \in \HH$ with covariance operator ${\bf C}^X$ of trace class, the claim follows selecting $n$ large enough.
\end{proof}

\begin{comment}
\begin{rem}
At first glance it appears surprising that the above bound only depends on the trace, and not any deeper structure of the underlying covariance operators. However, an heuristic explanation is that
\begin{align*}
\E\bigl[\|X\|_{\HH}^2\bigr] = \tr\bigl({\bf C}^{X}\bigr), \quad X \in \HH
\end{align*}
where ${\bf C}^X$ denotes the covariance operator of $X$, and this is the relevant moment that appears in the Taylor expansion in the proof.
\end{rem}
\end{comment}

Next, recall that
\begin{align*}
{\bf A}_{n,k}^c = \sum_{j < -n+k} \alpha_j + \sum_{j > k-1} \alpha_j \quad 1 \leq k \leq n,
\end{align*}
and that we have the decomposition
\begin{align}\nonumber
\sum_{k = 1}^n Z_k &= \sum_{k = 1}^n {\bf A}(\xi_k) - \sum_{k = 1}^n {\bf A}_{n,k}^c(\xi_k) + \sum_{k > n} {\mathbf A}_{n,k}(\xi_k) + \sum_{k < 1} {\mathbf A}_{n,k}(\xi_{k})\\& \stackrel{def}{=} \sum_{k = 1}^n {\bf A}(\xi_k) + {\bf III}_n + {\bf IV}_n + {\bf V}_n.
\end{align}

We now have our second lemma.
\begin{lem}\label{lem_approx_cov_operators_II}
Assume that Assumption \ref{ass_main} holds. Then
\begin{align*}
\Bigl\|{\bf A} {\bf C}^{\epsilon} {\bf A}^* - \frac{1}{n}\E\bigl[\langle S_1^{n+1}(Z), x \rangle S_1^{n+1}(Z) \bigr] \Bigr\|_{\HHS} \lesssim \frac{1}{n}\sum_{j \in \Z} (|j| \wedge n) \bigl\|\alpha_j\bigr\|_{\HHS}.
\end{align*}
\end{lem}

\begin{proof}[Proof of Lemma \ref{lem_approx_cov_operators_II}]
Observe that $\sum_{k = 1}^n {\bf A}(\xi_k) + {\bf III}_n$, ${\bf IV}_n$ and ${\bf V}_n$ are all mutually independent. It follows that for any $x \in \HH$ we have that
\begin{align}\label{eq_IV_V_cross}
\E\bigl[\langle {\bf IV}_n, x \rangle {\bf V}_n \bigr] = \E\bigl[\langle {\bf V}_n, x \rangle {\bf IV}_n \bigr] = 0,
\end{align}
(with $0 \in \HH$), and this remains valid if we substitute ${\bf IV}_n$ or ${\bf V}_n$ with $\sum_{k = 1}^n {\bf A}(\xi_k) + {\bf III}_n$. Similarly, if $i \neq j$ one readily derives that for $x \in \HH$
\begin{align}\label{eq_A_and_Ac_cross}
\E\bigl[\langle {\bf A}(\xi_i), x \rangle {\bf A}_{j}(\xi_j) \bigr] = \E\bigl[\langle {\bf A}(\xi_i), x \rangle {\bf A}_{n,j}^c(\xi_j) \bigr] = \E\bigl[\langle {\bf A}_{n,i}^c(\xi_i), x \rangle {\bf A}(\xi_j) \bigr] = 0,
\end{align}
(with $0 \in \HH$), exploiting the independence of $\xi_i$ and $\xi_j$ and the linearity of the operators ${\bf A}$ and ${\bf A}_{n,j}^c$. % use expansion in Hilbert space
Denote with $e_j$ the eigenfunctions of ${\bf C}^{\epsilon}$, and with $\lambda_j^{\epsilon}$ its corresponding eigenvalues. If $i = j$, it then follows that
\begin{align}\nonumber \label{eq_A_and_Ac}
\Bigl\|\E\bigl[\langle {\bf A}(\epsilon_i), x \rangle {\bf A}_{n,i}^c(\epsilon_i) \bigr]\Bigr\|_{\HHS} &=  \Bigl\|\sum_{l = 1}^{\infty}\lambda_l^{\epsilon}\E\bigl[\langle {\bf A}_{}(e_l), x \rangle {\bf A}_{n,i}^c(e_l) \bigr]\Bigr\|_{\HHS}\\ \nonumber &\leq \sum_{l = 1}^{\infty} \lambda_j^{\epsilon}\bigl\|{\bf A}_{}\bigr\|_{\HHS} \bigl\|x\bigr\|_{\HH} \bigl\|{\bf A}_{n,i}^c\bigr\|_{\HHS}\\&\lesssim \bigl\|x\bigr\|_{\HH} \bigl\|{\bf A}_{n,i}^c\bigr\|_{\HHS},
\end{align}
where we used Cauchy-Schwarz and $\sum_{j = 1}^{\infty} \|\alpha_j\|_{\HHS} < \infty$. As before, the same bound also applies if we exchange ${\bf A}_{n,i}^c$ and ${\bf A}$. Similarly, we also obtain that
\begin{align}\label{eq_A_and_A}
\Bigl\|\E\bigl[\langle {\bf A}_{n,i}(\epsilon_i), x \rangle {\bf A}_{n,i}(\epsilon_i) \bigr]\Bigr\|_{\HHS} \lesssim \bigl\|x\bigr\|_{\HH} \bigl\|{\bf A}_{n,i}\bigr\|_{\HHS}
\end{align}
and
\begin{align}\label{eq_Ac_and_Ac}
\Bigl\|\E\bigl[\langle {\bf A}_{n,i}^c(\epsilon_i), x \rangle {\bf A}_{n,i}^c(\epsilon_i) \bigr]\Bigr\|_{\HHS} \lesssim \bigl\|x\bigr\|_{\HH} \bigl\|{\bf A}_{n,i}^c\bigr\|_{\HHS}.
\end{align}
The treatment of ${\bf IV}_n$, ${\bf V}_n$ follows as in Lemma \ref{lem_control_Lambda_n+1}. For the sake of completeness, we have that
\begin{align}\nonumber \label{eq_lem_approx_cov_II_2}
\Bigl\|\E\bigl[\langle {\bf IV}_n, x \rangle {\bf IV}_n \bigr]\Bigr\|_{\HHS} &\leq \sum_{i > n} \sum_{j > n} \Bigl\|\E\bigl[\langle {\mathbf A}_{n,i}(\xi_i), x \rangle {\mathbf A}_{n,j}(\xi_j)\bigr]\Bigr\|_{\HHS} \\ \nonumber &= \sum_{i > n}  \Bigl\|\E\bigl[\langle {\mathbf A}_{n,i}(\xi_i), x \rangle {\mathbf A}_{n,i}(\xi_i)\bigr]\Bigr\|_{\HHS} \\&\lesssim  \sum_{i > n} \bigl\|{\bf A}_{n,i}\bigr\|_{\HHS} \bigl\|x\bigr\|_{\HH} \lesssim \sum_{j \in \Z} (|j|\wedge n) \bigl\|\alpha_i\bigr\|_{\HHS} \bigl\|x\bigr\|_{\HH}.
\end{align}
The same bound applies to ${\bf V}_n$, that is, we have
\begin{align}\label{eq_lem_approx_cov_II_3}
\Bigl\|\E\bigl[\langle {\bf V}_n, x \rangle {\bf V}_n \bigr]\Bigr\|_{\HHS} \lesssim \sum_{j \in \Z} (|j|\wedge n) \bigl\|\alpha_i\bigr\|_{\HHS} \bigl\|x\bigr\|_{\HH}.
\end{align}
By independence, we have
\begin{align}\nonumber
\E\bigl[\langle S_1^{n+1}(Z), x \rangle S_1^{n+1}(Z) \bigr] &= \E\bigl[\langle \sum_{k = 1}^n {\bf A}(\xi_k) + {\bf III}_n, x \rangle \bigl(\sum_{k = 1}^n{\bf A}(\xi_k) + {\bf III}_n\bigr)\bigr]  \\&+ \E\bigl[\langle {\bf IV}_n, x \rangle {\bf IV}_n \bigr] + \E\bigl[\langle {\bf V}_n, x \rangle {\bf V}_n \bigr].
\end{align}
Using \eqref{eq_A_and_Ac} and \eqref{eq_Ac_and_Ac}, it follows that
\begin{align}\nonumber \label{eq_lem_approx_cov_II_4}
&\Bigl\|\E\bigl[\langle \sum_{k = 1}^n {\bf A}(\xi_k) + {\bf III}_n, x \rangle \bigl(\sum_{k = 1}^n{\bf A}(\xi_k) + {\bf III}_n\bigr)\bigr] - \sum_{k = 1}^n\E\bigl[\langle {\bf A}(\xi_k), x \rangle {\bf A}(\xi_k)\bigr] \Bigr\|_{\HHS} \\&\lesssim \sum_{j \in \Z} (|j|\wedge n) \bigl\|\alpha_j\bigr\|_{\HHS}.
\end{align}
Combining \eqref{eq_lem_approx_cov_II_2}, \eqref{eq_lem_approx_cov_II_3} and \eqref{eq_lem_approx_cov_II_4}, we finally obtain that
\begin{align}
\Bigl\|\E\bigl[\langle S_1^{n+1}(Z), x \rangle S_1^{n+1}(Z) \bigr] - \sum_{k = 1}^n\E\bigl[\langle {\bf A}(\xi_k), x \rangle {\bf A}(\xi_k)\bigr] \Bigr\|_{\HHS} \lesssim \sum_{j \in \Z} (|j|\wedge n) \bigl\|\alpha_j\bigr\|_{\HHS}.
\end{align}
Since we have that
\begin{align*}
\E\bigl[\langle {\bf A}(\xi_k), x \rangle {\bf A}(\xi_k)\bigr] = {\bf A} {\bf C}^{\epsilon} {\bf A}^*,
\end{align*}
it follows that
\begin{align*}
\Bigl\|{\bf A} {\bf C}^{\epsilon} {\bf A}^* - \frac{1}{n}\E\bigl[\langle S_1^{n+1}(Z), x \rangle S_1^{n+1}(Z) \bigr] \Bigr\|_{\HHS} \lesssim \frac{1}{n}\sum_{j \in \Z} (|j| \wedge n) \bigl\|\alpha_j\bigr\|_{\HHS}.
\end{align*}
\end{proof}

We are now ready to proceed to the actual proof. For $\mu \in \HH$ denote with
\begin{align*}
Y_1 = \frac{1}{\sqrt{n}}\sum_{k = 1}^n {\bf A}(\xi_k) + \mu, \quad Y_2 = \frac{1}{\sqrt{n}}\sum_{k = 1}^n Z_k + \mu,
\end{align*}
and the corresponding covariance operators with ${\bf C}^{Y_1}$, ${\bf C}^{Y_2}$. Note that ${\bf C}^{Y_1} = {\bf A} {\bf C}^{\epsilon} {\bf A}^*$. The aim is to invoke Lemma \ref{lem_gauss_hoelder_inf}. To this end, we need to establish the necessary bounds. Since
\begin{align*}
\sum_{j \in \Z} (|j| \wedge n) \bigl\|\alpha_j\bigr\|_{\HHS} = \oo\bigl(n\bigr)
\end{align*}
due to $\sum_{j \in \Z} \|\alpha_j\|_{\HHS} < \infty$, Lemma \ref{lem_approx_cov_operators_II} yields that
\begin{align*}
\bigl\|{\bf C}^{Y_1} - {\bf C}^{Y_2} \bigr\|_{\HHS} = \oo\bigl(1\bigr).
\end{align*}
Hence condition \eqref{eq_condi_lem_gauss_hoelder_inf} is valid by Assumption . Next, note that by the independence of $\sum_{k = 1}^n {\bf A}(\xi_k) + {\bf III}_n$, ${\bf IV}_n$ and ${\bf V}_n$, we have that
\begin{align}\label{eq_thm_be_II_2}
n\E\bigl[\bigl\|Y_1\bigr\|_{\HH}^2\bigr] = \sum_{k = 1}^n \E\bigl[\|{\bf A}(\xi_k) - {\bf A}_{n,k}^c(\xi_k)\|_{\HH}^2\bigr] + \E\bigl[\|{\bf IV}_n\|_{\HH}^2\bigr] + \E\bigl[\|{\bf V}_n\|_{\HH}^2\bigr].
\end{align}
Using the triangle inequality and $a^2-b^2 = (a-b)(a+b)$, we get that
\begin{align*}
\bigl|\E\bigl[\|{\bf A}(\xi_k) - {\bf A}_{n,k}^c(\xi_k)\|_{\HH}^2 - \|{\bf A}(\xi_k)\|_{\HH}^2\bigr] \bigr| &\leq \E\bigl[\|{\bf A}_{n,k}^c(\xi_k)\|_{\HH}\bigl(2 \|{\bf A}(\xi_k)\|_{\HH} + \|{\bf A}_{n,k}^c(\xi_k)\|_{\HH}\bigr)\bigr] \\&\lesssim \|{\bf A}_{n,k}^c(\xi_k)\|_{\HHS} \E\bigl[\|\xi_0\|_{\HH}^2\bigr].
\end{align*}
Hence we obtain the estimate
\begin{align}\label{eq_thm_be_II_3}
\Bigl|\sum_{k = 1}^n \E\bigl[\|{\bf A}(\xi_k) - {\bf A}_{n,k}^c(\xi_k)\|_{\HH}^2\bigr] \Bigr| \lesssim  \sum_{j \in \Z} (|j| \wedge n) \bigl\|\alpha_j\bigr\|_{\HHS}.
\end{align}
Similarly, proceeding as in Lemma \ref{lem_control_Lambda_n+1}, one readily computes that for $p \in (2,3]$
\begin{align}\label{eq_thm_be_II_4}
\bigl\|\|{\bf IV}_n\|_{\HH}\bigr\|_p^p,\bigl\|\|{\bf V}_n\|_{\HH}\bigr\|_p^p  \lesssim  \sum_{j \in \Z} (|j| \wedge n) \bigl\|\alpha_j\bigr\|_{\HHS} \bigl\|\|\epsilon_0\|_{\HH}\bigr\|_p.
\end{align}
Combining \eqref{eq_thm_be_II_2} with \eqref{eq_thm_be_II_3} and \eqref{eq_thm_be_II_4}, the claim then follows from Lemma \ref{lem_gauss_hoelder_inf}.
\end{proof}

%\subsection{Proof of Theorem \ref{thm_lower_bound}}

\begin{proof}[Proof of Theorem \ref{thm_lower_bound}]
For the proof, we construct an example where the upper bound is obtained, up to a constant. It suffices to consider the special case where $\HH = \R$ and $X_k$ is 'purely' non-causal, that is, $X_k = \sum_{j = 0}^{\infty} \alpha_j \epsilon_{k+j}$, $\alpha_j \in \R$. Moreover, we assume throughout this section that $\E\bigl[\epsilon_k^2\bigr] = 1$, $\alpha_j \geq 0$ and ${\bf A} = 1$ to simplify matters. We first require the following Lemma.

\begin{lem}\label{lem_control_sum_alpha_I}
Assume that $\alpha_j \geq 0$ such that $\sum_{j \in \N} \alpha_j < \infty$ and $\sum_{j \in \N} (j \wedge n) \alpha_j \to \infty$ as $n \to \infty$. Then
\begin{align*}
\sum_{k > n}^{} \bigl({\bf A}_{n,k}\bigr)^2 + \sum_{k < 1}^{} \bigl({\bf A}_{n,k}\bigr)^2&= \oo\Bigl(\sum_{j \in \N} (j \wedge n) \alpha_j\Bigr),\\
\sum_{k = 1}^n \bigl({\bf A}_{n,k}^c \bigr)^2 &= \oo\Bigl(\sum_{j \in \N} (j \wedge n) \alpha_j\Bigr).
\end{align*}
\end{lem}

\begin{proof}[Proof of Lemma \ref{lem_control_sum_alpha_I}]
We only show the first claim, the second  follows in an analogue manner. Since $\sum_{j > L} \alpha_j \to 0$ as $L \to \infty$, there exists $\delta_n \to 0$ and $m_n \to \infty$ as $n \to \infty$, such that
\begin{align}\nonumber
\sum_{k > n}^{} \bigl({\bf A}_{n,k}\bigr)^2 &= \sum_{k >n}^{} \Bigl(\sum_{j = -n+k}^{k-1} \alpha_j \Bigr)^2 = \sum_{l = 1}^{\infty} \Bigl(\sum_{j = l}^{n+l-1} \alpha_j  \Bigr)^2 \\&\leq \sum_{l = 1}^{m_n} \sum_{j = l}^{n+l-1} \alpha_j + \delta_n \sum_{l > m_n} \sum_{j = l}^{n+l-1} \alpha_j \leq \sum_{j = 1}^{m_n} (j \wedge n) \alpha_j + \delta_n \sum_{j \in \N} (j \wedge n) \alpha_j,
\end{align}
where we also used $\sum_{j \in \N} \alpha_j = 1$. Since $\sum_{j \in \N} (j \wedge n) \alpha_j \to \infty$, we can choose $m_n$ such that
\begin{align*}
\sum_{j = 1}^{m_n} (j \wedge n) \alpha_j = \oo\Bigl(\sum_{j \in \N} (j \wedge n) \alpha_j\Bigr),
\end{align*}
and the claim follows for $\sum_{k > n}^{} \bigl({\bf A}_{n,k}\bigr)^2$. Regarding expression $\sum_{k < 1}^{} \bigl({\bf A}_{n,k}\bigr)^2$, note that ${\bf A}_{n,k} = 0$ for $k \leq 0$ by assumption, hence the claim.
\end{proof}

It is known in the literature that the rate $n^{p/2 - 1}$ is optimal (cf. ~\cite{petrov_book_1995}). Hence due to Theorem \ref{thm_berry_esseen}, it suffices to derive a lower bound for
\begin{align*}
\sup_{x \in \R}\Bigl|P\bigl(Z_{{\bf \Lambda}_n} \leq x \bigr) - P\bigl(Z_{\bf \Lambda} \leq x \bigr)\Bigr|.
\end{align*}
Moreover, we may assume without loss of generality that
\begin{align}
\sum_{j \in \N} (j \wedge n) \alpha_j \to \infty \quad \text{as $n \to \infty$,}
\end{align}
since otherwise the claim immediately follows. Denote with $\sigma^2 = \bigl\|Z_{\bf \Lambda}\bigr\|_2^2$ and $\sigma_n^2 = \bigl\|Z_{{\bf \Lambda}_n}\bigr\|_2^2$. Note that $\sigma^2 = {\bf A}^2$, and by \eqref{eq_thm_be_II_2} we have $\sigma_n^2 = n^{-1} \sum_{k \in \N} {\bf A}_{n,k}^2$. The proof relies on the following lower bound. For large enough $n$, there exists a constant $C_{\alpha} > 0$ (which can be chosen arbitrarily smaller than two) such that
\begin{align}\label{eq_fact_alpha_1}
\sigma^2 - \sigma_n^2 \geq \frac{C_{\alpha}}{n}\sum_{j \in \N} (j \wedge n) \alpha_j.
\end{align}
We first derive this lower bound, a simple application of the mean value Theorem then yields the claim, see below.
By Lemma \ref{lem_control_sum_alpha_I} it follows that
\begin{align*}
\sigma_n^2 &= \frac{1}{n} \sum_{k = 1}^n {\bf A}_{n,k}^2 + \frac{1}{n} \sum_{k > n}{\bf A}_{n,k}^2 + \frac{1}{n} \sum_{k < 1}^{} \bigl({\bf A}_{n,k}\bigr)^2 \\&= \frac{1}{n} \sum_{k = 1}^n {\bf A}_{n,k}^2 + \oo\Bigl(\sum_{j \in \N} (j \wedge n) \alpha_j\Bigr).
\end{align*}
On the other hand, we also have
\begin{align}\nonumber
{\bf A}^2 - \sum_{k = 1}^n \bigl({\bf A}_{n,k}\bigr)^2 &=  \sum_{k = 1}^n {\bf A}_{n,k}^c\bigl({\bf A} + {\bf A}_{n,k}\bigr) = \sum_{k = 1}^n {\bf A}_{n,k}^c\bigl(2{\bf A} - {\bf A}_{n,k}^c\bigr) \\&=  2\sum_{k = 1}^n {\bf A}_{n,k}^c - \sum_{k = 1}^n \bigl({\bf A}_{n,k}^c\bigr)^2.
\end{align}
Hence another application of Lemma \ref{lem_control_sum_alpha_I} yields the claim. We now finalize the proof. Let $0< x < \infty$ and denote with ${\mathcal S} = [x/\sigma^2,2x/\sigma^2]$. Then for large enough $n$, it follows from the mean value Theorem that
\begin{align*}
P\bigl(Z_{{\bf \Lambda}_n} \leq x \bigr) - P\bigl(Z_{\bf \Lambda} \leq x \bigr) \geq x \Big(\frac{1}{\sqrt{\sigma_n^2}} - \frac{1}{\sqrt{\sigma^2}}\Big) \inf_{y \in {\mathcal S}}\phi(y),
\end{align*}
where $\phi(y)$ denotes the density function of the Gaussian standard distribution. Using the fact that $\sigma^2 = 1$, we further obtain
\begin{align*}
P\bigl(Z_{{\bf \Lambda}_n} \leq x \bigr) - P\bigl(Z_{\bf \Lambda} \leq x \bigr) \geq C_x \bigl(\sigma^2 - \sigma_n^2\bigr)
\end{align*}
for some $C_x > 0$. Employing the lower bound of \eqref{eq_fact_alpha_1} then yields the desired result.
\end{proof}

\subsection{Proofs of Section \ref{sec_main_non_lin}}\label{sec_proof_nonlin}

\begin{proof}[Proof of Theorem \ref{thm_non_lin}]
The main idea of the proof is based on a conditioning argument, similar in spirit to the approach in ~\cite{jirak_aop_2015}. To this end, we first require some notation. Put $n = 2 K L$ for $L \thicksim n$ and $3 \leq K < \infty$, $K \in \N$ to be specified later. To simplify the exposition, we also assume that $L \in \N$, see the very last comment at the end of the proof on how to remove this assumption. We make the convention that $X_k = 0$ for $k \not \in \{1,\ldots,n\}$. Put $\mathcal{I}_l = \{k: \, (l-1)K < k \leq l K \}$, $\mathcal{I}_l^* = \mathcal{I}_l \cup \{(l-1)K \}$, $\F = \mathcal{F}_{L}^{(e)} = \sigma(\epsilon_k, \, k \in \mathcal{I}_{2l}^*, \, 1 \leq l \leq  L)$, and introduce the block variables
\begin{align*}
V_l = \sum_{k \in \mathcal{I}_l} X_k, \quad \overline{V}_l = V_l - \E\bigl[V_l| \mathcal{F}_L^{(e)}\bigr].
\end{align*}
The fact that $\mathcal{I}_l^*$ also contains the left endpoint of the interval (unlike to $\mathcal{I}_l$) is important in the sequel. We denote the corresponding even (e) and odd (o) partial sums with
\begin{align}\nonumber
S_L^{(o)}(\overline{V}) &= n^{-1/2}\sum_{l = 0}^{L} \overline{V}_{2l+1}, \quad
S_L^{(e)}({V}) = n^{-1/2}\sum_{l = 1}^{L} {V}_{2l}, \\
R_L^{(o)}(V) &= n^{-1/2}\sum_{l = 0}^{L} \E[V_{2l+1}|\mathcal{F}_L^{(e)}].
\end{align}
Hence we have the decomposition
\begin{align*}
n^{-1/2}S_n(X) = S_L^{(o)}(\overline{V}) + S_L^{(e)}({V}) + R_L^{(o)}(V),
\end{align*}
where we used $X_k = 0$ for $k \not \in \{1,\ldots,n\}$. Next, consider the conditional probability measure $P_{|\F}(\cdot) = P(\cdot |\mathcal{F}_L^{(e)})$. Observe that $\{\overline{V}_{2l+1}\}_{0 \leq l \leq L}$ is a sequence of centered, independent random variables under $P_{|\F}$ since $K \geq 3$ and due to the inclusion of the left endpoint in $\mathcal{I}_l^*$. Also note that $S_L^{(e)}({V})$ and $R_L^{(o)}(V)$ are $\mathcal{F}_L^{(e)}$-measurable. We make heavy use of these properties in the sequel. Likewise, under the measure $P_{|\F}$, let $Z_{L|\F}^{(o)}$ be a zero mean Gaussian random variable with (conditional) covariance operator
\begin{align}
{\bf \Lambda}_{|\F}^{(o)}(\cdot) = n^{-1} \sum_{l = 0}^{L}\E[\overline{V}_{2l+1} \langle \overline{V}_{2l+1},\cdot \rangle | \mathcal{F}_L^{(e)}].
\end{align}
Similarly, under $P$, let $Z_{L}^{(o)}$ and $Z_L^{(e,o)}$ be two mutually independent, zero mean Gaussian random variables, independent of $\mathcal{F}_L^{(e)}$, with covariance operators
\begin{align}\nonumber
{\bf \Lambda}_{}^{(o)}(\cdot) &= \E\bigl[\langle S_L^{(o)}(\overline{V}), \cdot \rangle S_L^{(o)}(\overline{V})  \bigr] = \E\bigl[{\bf \Lambda}_{|\F}^{(o)}(\cdot)\bigr],\\
{\bf \Lambda}_{}^{(e)}(\cdot) &= \E\bigl[\langle S_L^{(e)}(\overline{V}) + R_L^{(o)}(V), \cdot \rangle \bigl(S_L^{(e)}(\overline{V}) + R_L^{(o)}(V)\bigr) \bigr].
\end{align}

The proof relies on the following decomposition
\begin{align}
\Delta_n(\mu) \leq \E\bigl[\sup_{x \in \R}\bigl({\bf I}_L(x) + {\bf II}_L(x)\bigr)\bigr] + \sup_{x\in \R}\big|\E\bigl[{\bf III}_L(x)\bigr]\bigr| + \sup_{x\in \R}{\bf IV}_L(x),
\end{align}
where
\begin{align}\nonumber
{\bf I}_L(x) &= \Big|P_{|\F}\Big(\big\|S_L^{(o)}(\overline{V}) + S_L^{(e)}({V}) + R_L^{(o)}(V) + \mu\big\|_{\HH}\leq x \Big) \\ \nonumber &- P_{|\F}\Big(\big\|Z_{L|\F}^{(o)} + S_L^{(e)}({V}) + R_L^{(o)}(V) + \mu\big\|_{\HH}\leq x \Big) \Big|,\\ \nonumber
{\bf II}_L(x) &= \Big|P_{|\F}\Big(\big\|Z_{L|\F}^{(o)} + S_L^{(e)}({V}) + R_L^{(o)}(V) + \mu\big\|_{\HH}\leq x \Big) \\ \nonumber &- P_{|\F}\Big(\big\|Z_{L}^{(o)} + S_L^{(e)}({V}) + R_L^{(o)}(V) + \mu\big\|_{\HH}\leq x \Big) \Big|,\\ \nonumber
{\bf III}_L(x) &= P_{|\F}\Big(\big\|Z_{L}^{(o)} + S_L^{(e)}({V}) + R_L^{(o)}(V) + \mu\big\|_{\HH}\leq x \Big) \\ \nonumber &- P_{|\F}\Big(\big\|Z_{L}^{(o)} + Z_L^{(e,o)} + \mu\big\|_{\HH}\leq x \Big),\\
{\bf IV}_L(x) &= \Big| P_{|\F}\Big(\big\|Z_{L}^{(o)} + Z_L^{(e,o)} + \mu\big\|_{\HH}\leq x \Big) - P\Big(\big\|Z_{\Lambda} + \mu\big\|_{\HH}\leq x \Big)\Big|.
\end{align}

Below, we derive separate bounds for all four quantities. The key step is dealing with ${\bf II}_L(x)$, where the dependence gets disentangled asymptotically to independence.

{\bf Case ${\bf I}_L(x)$}: Here we apply Lemma \ref{lem_ulyanov} under the conditional probability $P_{|\F}$, which makes all involved quantities random. To ensure applicability, we need to rule out any pathologies in advance. In particular, we need to control the eigenvalues of the random operator ${\bf \Lambda}_{|\F}^{(o)}$.
\begin{comment}
\begin{align*}
{\bf C}_{|\F}(\cdot) = n^{-1} \sum_{l = 0}^{L}\E[\overline{V}_{2l+1} \langle \overline{V}_{2l+1},\cdot \rangle | \mathcal{F}_L^{(e)}].
\end{align*}
\end{comment}
To this end, let $\mathcal{C}_{\delta} = \bigl\{ \|{\bf \Lambda}_{|\F}^{(o)} - {\bf \Lambda}_{}\|_{\HHS} \leq \delta \bigr\}$, $\delta > 0$, and put $\mathcal{I}_l' = \{k: \, (l-1)K + 1 < k \leq (l K -1)\}$, $1 \leq l \leq 2L$. Then we have by independence
\begin{align}\label{eq_equaliy_means_condi}
\E\bigl[X_k| \mathcal{F}_L^{(e)}\bigr] = \E\bigl[X_k\bigr], \quad k \in \mathcal{I}_{2l+1}'.
\end{align}
Using \eqref{eq_equaliy_means_condi}, routine calculations reveal that
%need cauchy-schwarz and four moments
\begin{align*}
\bigl\|\E[\overline{V}_{2l+1} \langle \overline{V}_{2l+1},\cdot \rangle] - \E[{V}_{2l+1} \langle {V}_{2l+1},\cdot \rangle] \bigr\|_{\HHS} < \infty
\end{align*}
since $K < \infty$. Applying this bound then leads to
\begin{align}\label{eq_thm_non_lin_5}
\bigl\|{\bf \Lambda}_{}^{(o)} -  {\bf \Lambda}_{}^{} \bigr\|_{\HHS} \lesssim \frac{L \E\bigl[\|X_0\|_{\HH}^2\bigr]}{n} \lesssim \frac{\E\bigl[\|X_0\|_{\HH}^2\bigr]}{K}.
\end{align}
Selecting $K$ sufficiently large (but finite), we thus obtain
\begin{align}
\mathcal{C}_{\delta}^c = \bigl\{\|{\bf \Lambda}_{|\F}^{(o)} - {\bf \Lambda}_{}\|_{\HHS} > \delta \bigr\} \subseteq \bigl\{ \|{\bf \Lambda}_{|\F}^{(o)} - {\bf \Lambda}_{}^{(o)}\|_{\HHS} > \delta/2 \bigr\},
\end{align}
and hence by Markovs inequality
\begin{align*}
P\bigl(\mathcal{C}_{\delta}^c\bigr) \leq (\delta/2)^{-2}\E\bigl[\|{\bf \Lambda}_{|\F}^{(o)} - {\bf \Lambda}_{}^{(o)}\|_{\HHS}^{2}\bigr].
\end{align*}
Denote with $\|\cdot\|_{\HSN}$ the Hilbert-Schmidt norm. Since $\|\cdot\|_{\HHS} \leq \|\cdot\|_{\HSN}$ and the space of Hilbert-Schmidt operators forms again a Hilbert space, using a Rosenthal inequality for Hilbert space valued sequences (cf. ~\cite{junge2003}), we obtain
\begin{align*}
&n\E[\|{\bf \Lambda}_{|\F}^{(o)} - {\bf \Lambda}_{}^{(o)}\|_{\HHS}^2\bigr] \\&\lesssim \sum_{l = 0}^{L} \E\bigl[\bigl\|\E[\overline{V}_{2l+1} \langle \overline{V}_{2l+1},\cdot \rangle | \mathcal{F}_L^{(e)}] - \E[\overline{V}_{2l+1} \langle \overline{V}_{2l+1},\cdot \rangle] \bigr\|_{\HSN}^{2} \bigr]\\&\lesssim \sum_{l = 0}^{L} \E\bigl[\bigl\|\overline{V}_{2l+1}\bigr\|_{\HH}^{2}\bigr] \lesssim \sum_{l = 0}^{L} \E\bigl[\bigl\|\sum_{k \in \mathcal{I}_l} X_k \bigr\|_{\HH}^{4}\bigr] \lesssim L \E\bigl[\|X_0\|_{\HH}^4\bigr].
\end{align*}
This together with the above yields $P(\mathcal{C}_{\delta}^c) \lesssim (\delta^2 L)^{-1} \E\bigl[\|X_0\|_{\HH}^4\bigr]$. Next, let
\begin{align*}
T_l = \E_{}\bigl[\|\overline{V}_{2l+1}\|_{\HH}^2 \ind(\|\overline{V}_{2l+1} \|_{\HH} \geq n^{1/2}) + n^{-1/2}\|\overline{V}_{2l+1}\|_{\HH}^3 \bigl|\mathcal{F}_L^{(e)} \bigr],
\end{align*}
where routine calculations reveal that
\begin{align}\label{eq_thm_two_dep_4}
\E\bigl[T_l\bigr] \lesssim n^{-1/2} \E\bigl[\|\overline{V}_{2l+1}\|_{\HH}^3\bigr] \lesssim n^{-1/2}\E\bigl[\|X_0\|_{\HH}^3\bigr].
\end{align}
Note that $\{T_l\}_{0 \leq l \leq L}$ is a sequence of independent, real valued random variables. Denote with
\begin{align}
\mathcal{D}_{} = \Bigl\{ \Bigl|\sum_{l = 0}^{L} \bigl(T_l - \E[T_l]\bigr)\Bigr| \leq L^{1/2} \Bigr\}.
\end{align}
Using Burkholders, triangle and Jensens inequality, we get
%use jensen two times, and \E[|V|\ind(V \geq N^{1/2})] \leq N^{-q/2+1}\E[|V|^q]
\begin{align*}
\Big\|\sum_{l = 0}^{L} \bigl(T_l - \E[T_l]\bigr)\Bigr\|_{3/2}^{3/2} \lesssim \sum_{l = 0}^{L} \bigl\|T_l - \E[T_l]\bigr\|_{3/2}^{3/2} \lesssim L n^{-3/4} \E\bigl[\|X_0\|_{\HH}^{9/2}\bigr].
\end{align*}
Hence we conclude via Markovs inequality
\begin{align}\label{eq_thm_two_dep_5}
P\bigl(\mathcal{D}^c\bigr) \lesssim L^{1 - 3/4} n^{-3/4}\E\bigl[\|X_0\|_{\HH}^{9/2}\bigr] \lesssim L^{-1/2}\E\bigl[\|X_0\|_{\HH}^{9/2}\bigr].
\end{align}
We are now in position to derive the actual bound. Observe that
\begin{align}
\E\bigl[\sup_{x \in \R}{\bf I}_L(x)\bigr] \leq 2P\bigl(\mathcal{C}_{\delta}^c\bigr) + 2P\bigl(\mathcal{D}_{}^c\bigr) + \E\bigl[\sup_{x \in \R}{\bf I}_L(x)\ind_{\mathcal{C}_{\delta} \cap \mathcal{D}}\bigr].
\end{align}
Since $S_L^{(e)}({V}),R_L^{(o)}(V) \in \mathcal{F}_L^{(e)}$ and $\{\overline{V}_l\}_{0 \leq l \leq L}$ is a sequence of independent zero mean random variables under $P_{|\F}$, applying Lemma \ref{lem_ulyanov} leads to
\begin{align}
&\sup_{x \in \R}{\bf I}_L(x)\ind_{\mathcal{C}_{\delta} \cap \mathcal{D}} \leq C_{|\F}\ind_{\mathcal{C}_{\delta} \cap \mathcal{D}}\bigl(1 + \|S_L^{(e)}({V}) + R_L^{(o)}(V) + \mu\|_{\HH}^3\bigr) n^{-1}\sum_{l = 0}^{L} T_l.
\end{align}

\begin{comment}
\begin{align}
&\sup_{x \in \R}{\bf I}_L(x)\ind_{\mathcal{C}_{\delta} \cap \mathcal{D}} \leq C_{|\F}\ind_{\mathcal{C}_{\delta} \cap \mathcal{D}}\bigl(1 + \|S_L^{(e)}({V}) + R_L^{(o)}(V) + \mu\|_{\HH}^3\bigr)\\& \times \sum_{l = 0}^{L}\E_{|\F}\bigl[n^{-1}\|\overline{V}_{2l+1}\|_{\HH}^2 \ind(\|\overline{V}_{2l+1} \|_{\HH} \geq n^{1/2}) + n^{-3/2}\|\overline{V}_{2l+1}\|_{\HH}^3\bigr].
\end{align}
\end{comment}

Since $\lambda_{13} > 0$, selecting $\delta > 0$ sufficiently small (and $K = K_{\delta}$ sufficiently large) and using Lemma \ref{lem_eigen_gen_upper_bound}, we may bound $C_{|\F}\ind_{\mathcal{C}_{\delta}} \leq C_{\lambda}$, where $C_{\lambda}$ only depends on $\lambda_{13}$. In addition, by construction of the set $\mathcal{D}$, we have
\begin{align*}
&\E\Big[\ind_{\mathcal{D}}\bigl(1 + \|S_L^{(e)}({V}) + R_L^{(o)}(V) + \mu\|_{\HH}^3\bigr) n^{-1}\sum_{l = 0}^{L} T_l\Bigr] \\&\lesssim \E\bigl[1 + \|S_L^{(e)}({V}) + R_L^{(o)}(V) + \mu\|_{\HH}^3\bigr]\Big(n^{-1}\sum_{l = 0}^{L} \E\bigl[T_l\bigr] + L^{-1/2}\Big).
\end{align*}
Using \eqref{eq_thm_two_dep_4} and Rosenthals inequality, the above is further bounded by $L^{-1/2}\bigl(1 + \|\mu\|_{\HH}^3\bigr)\E\bigl[\|X_0\|_{\HH}^{3}\bigr]$, and hence we obtain
\begin{align}
\E\Big[\ind_{\mathcal{D}}\bigl(1 + \|S_L^{(e)}({V}) + R_L^{(o)}(V) + \mu\|_{\HH}^3\bigr) n^{-1}\sum_{l = 0}^{L} T_l\bigr] \lesssim L^{-1/2}\bigl(1 + \|\mu\|_{\HH}^3\bigr)\E\bigl[\|X_0\|_{\HH}^{3}\bigr].
\end{align}
We thus conclude from \eqref{eq_thm_two_dep_5}
\begin{align}\label{eq_thm_two_dep_7}
\E\bigl[\sup_{x \in \R}{\bf I}_L(x)\bigr] \lesssim L^{-1/2}\bigl(1 + \|\mu\|_{\HH}^3\bigr)\E\bigl[\|X_0\|_{\HH}^{9/2}\bigr].
\end{align}

{\bf Case ${\bf II}_L(x)$}:
\begin{comment}
\begin{align*}
&\E\Bigl[\Bigl\|\sum_{l = 0}^{\lfloor L/2 \rfloor}\bigl( \E[\overline{V}_{2l+1} \langle \overline{V}_{2l+1},\cdot \rangle | \mathcal{F}_L^{(e)}] - \E[\overline{V}_{2l+1} \langle \overline{V}_{2l+1},\cdot \rangle] \bigr) \Bigr\|_{\HHS}^{p/2} \Bigr] \\&\leq \E\Bigl[\Bigl\|\sum_{l = 0}^{\lfloor L/2 \rfloor}\bigl( \E[\overline{V}_{2l+1} \langle \overline{V}_{2l+1},\cdot \rangle | \mathcal{F}_L^{(e)}] - \E[\overline{V}_{2l+1} \langle \overline{V}_{2l+1},\cdot \rangle] \bigr) \Bigr\|_{\HSN} \Bigr] \leq
\end{align*}
\end{comment}
%Denote with ${\bf C}_{\xi|\F}(\cdot) = n^{-1} \sum_{l = 0}^{L}\E[\overline{V}_{2l+1} \langle \overline{V}_{2l+1},\cdot \rangle | \mathcal{F}_L^{(e)}]$.
Under the measure $P_{|\F}$, let $\bigl\{\xi_{k|\F}\bigr\}_{1 \leq k \leq n}$ be a sequence of independent, zero mean Gaussian random variables, where each $\xi_{k|\F}$ has covariance operator ${\bf \Lambda}_{|\F}^{(o)}$ (instead of sample size $n$ we could also select $N > n$, but $n$ is sufficient). Similarly, let  $\bigl\{\xi_{k}\bigr\}_{1 \leq k \leq n}$ be IID Gaussian sequences independent of $\mathcal{F}_L^{(e)}$, with covariance operator ${\bf \Lambda}_{}^{(o)}$. Next, introduce the mixed partial sums
\begin{align}
\sqrt{n}W_{k|\F} = \sum_{j = 1}^{k-1} \xi_{k|\F} + \sum_{j = k + 1}^n \xi_{k}, \quad 1 \leq k \leq n,
\end{align}
and the corresponding conditional covariance operators
\begin{align}\nonumber
\tilde{\bf \Lambda}_{k|\F}(\cdot) &= \E\bigl[\langle W_{k|\F}, \cdot \rangle W_{k|\F} |\mathcal{F}_L^{(e)} \bigr]\\&=\frac{k-1}{n}{\bf \Lambda}_{|\F}^{(o)}(\cdot) + \frac{n-k}{n}{\bf \Lambda}^{(o)}(\cdot),
\end{align}
with eigenvalues $\tilde{\lambda}_{k,j|\F}$, $j \in \N$. Since these are random, we need to control them as in the previous case ${\bf I}_L(x)$. To this end, define the set
\begin{align}
\tilde{\mathcal{C}}_{\delta} = \bigl\{\max_{1 \leq k \leq n}\|\tilde{\bf \Lambda}_{k|\F} - {\bf \Lambda}^{}\|_{\HHS} \geq \delta \bigr\}.
\end{align}
\begin{comment}
Denote with $\mathcal{I}_l' = \{k: \, ((l-1)K + 1)m < k < (l K -1)m\}$, $1 \leq l \leq L$. Then we have by independence
\begin{align}\label{eq_equaliy_means_condi}
\E\bigl[X_k| \mathcal{F}_L^{(e)}\bigr] = \E\bigl[X_k\bigr], \quad k \in \mathcal{I}_l'.
\end{align}
\end{comment}
Using \eqref{eq_equaliy_means_condi} and \eqref{eq_thm_non_lin_5} yields
%need cauchy-schwarz and four moments
\begin{align*}
\max_{1 \leq k \leq n}\bigl\|\E[\tilde{\bf \Lambda}_{k|\F}] - {\bf \Lambda} \bigr\|_{\HHS} \leq  \bigl\|{\bf \Lambda}^{(o)} - {\bf \Lambda} \bigr\|_{\HHS} + \bigl\|{\bf \Lambda}\bigr\|_{\HHS}/n \lesssim \frac{L \E\bigl[\|X_0\|_{\HH}^{2}\bigr]}{n} = \frac{\E\bigl[\|X_0\|_{\HH}^{2}\bigr]}{K}.
\end{align*}
Proceeding similarly as in the treatment of ${\mathcal{C}}_{\delta}^c$ and using Burkholders inequality for Hilbert-space valued martingales (cf. ~\cite{junge2003}), it follows that for large enough $K$
\begin{align}\label{eq_thm_non_lin_6}
P\bigl(\tilde{\mathcal{C}}_{\delta}^c \bigr)\lesssim (\delta^2 L)^{-1} \E\bigl[\|X_0\|_{\HH}^{4}\bigr], \quad \text{$\delta > 0$ sufficiently small.}
\end{align}
Next, following the same approach as in Lemma \ref{lem_gauss_hoelder_inf} and using similar arguments as in the previous case ${\bf I}_L(x)$, we obtain from \eqref{eq_thm_non_lin_6} that
\begin{align}\nonumber \label{eq_thm_non_lin_7}
\E\bigl[\sup_{x \in \R}{\bf II}_L(x)\bigr] &\leq 2P\bigl(\tilde{\mathcal{C}}_{\delta}^c\bigr) + \E\bigl[\sup_{x \in \R}{\bf II}_L(x)\ind_{\tilde{\mathcal{C}}_{\delta}}\bigr]\\& \lesssim L^{-1/2}\bigl(1 + \|\mu\|_{\HH}^3\bigr)\E\bigl[\|X_0\|_{\HH}^{4}\bigr] +  \E\Big[\sum_{k = 1}^n \int_{\R} \frac{Q_{k|\F}(t)}{t }  d\,t \ind_{\tilde{\mathcal{C}}_{\delta}}\Big],
\end{align}
where
\begin{align*}
Q_{k|\F}(t) = |t|\bigl|\E\bigl[\exp(\ic t \| W_{k|\F} \|_{\HH}^2) | \mathcal{F}_L^{(e)} \bigr]\bigl(\E[\|\xi_{k|\F}\|_{\HH}^2|\mathcal{F}_L^{(e)} ] - \E[\|\xi_{k}\|_{\HH}^2|\mathcal{F}_L^{(e)}]\bigr) \bigr|, \quad t \in \R.
\end{align*}
By Lemma \ref{lem_gauss_char}, we have the bound
\begin{align}\label{eq_lem_gauss_hoelder_inf__condi_4}
\bigl|\E\bigl[\exp(\ic t \| W_{k|\F} \|_{\HH}^2) |\mathcal{F}_L^{(e)} \bigr]\bigr|\ind_{\tilde{\mathcal{C}}_{\delta}} \leq \prod_{j = 1}^{\infty} \bigl(1 + 4 t^2 \tilde{\lambda}_{k,j|\F}^2 \bigr)^{-1/4}\ind_{\tilde{\mathcal{C}}_{\delta}},
\end{align}
where we recall that $\tilde{\lambda}_{k,j|\F}$ denote the eigenvalues of the covariance operator $\tilde{\bf \Lambda}_{k|\F}$. In particular, they are bounded away from zero uniformly on the set $\tilde{\mathcal{C}}_{\delta}$ for $\delta > 0$ sufficiently small due to Lemma \ref{lem_eigen_gen_upper_bound}. Observe next that by independence of $\{\xi_k\}_{k \in \Z}$ from $\mathcal{F}_L^{(e)}$
\begin{align*}%\label{eq_thm_non_lin_8}
\bigl|\E[\|\xi_{k|\F}\|_{\HH}^2|\mathcal{F}_L^{(e)}] - \E[\|\xi_{k}\|_{\HH}^2 |\mathcal{F}_L^{(e)}]\bigr| = n^{-1}\Bigl|\sum_{l = 0}^{L}\bigl(\E\bigl[\|\overline{V}_{2l+1}\|_{\HH}^2| \mathcal{F}_L^{(e)}\bigr] - \E\bigl[\|\overline{V}_{2l+1}\|_{\HH}^2\bigr]\bigr)\Bigr|.
\end{align*}
Since $\{\E\bigl[\|\overline{V}_{2l+1}\|_{\HH}^2| \mathcal{F}_L^{(e)}\bigr]\}_{0 \leq l \leq L}$ is a sequence of independent, real-valued random variables, Jensens and Rosenthals inequality yield
\begin{align}\label{eq_thm_non_lin_9}
\E\bigl[\bigl|\E[\|\xi_{k|\F}\|_{\HH}^2|\mathcal{F}_L^{(e)}] - \E[\|\xi_{k}\|_{\HH}^2 |\mathcal{F}_L^{(e)}]\bigr|\bigr] \lesssim L^{-1/2} \E\bigl[\|X_0\|_{\HH}^{4}\bigr]^{1/2}.
\end{align}
Combining \eqref{eq_lem_gauss_hoelder_inf__condi_4} with \eqref{eq_thm_non_lin_9}, we thus obtain
\begin{align}
\E\Big[\sum_{k = 1}^n \int_{\R} \frac{Q_{k|\F}(t)}{t }  d\,t \ind_{\tilde{\mathcal{C}}_{\delta}}\Big] \lesssim L^{-1/2} \E\bigl[\|X_0\|_{\HH}^{4}\bigr].
\end{align}
Together with \eqref{eq_thm_non_lin_7}, this yields the final bound
\begin{align}
\E\bigl[\sup_{x \in \R}{\bf II}_L(x)\bigr] \lesssim L^{-1/2} \E\bigl[\|X_0\|_{\HH}^{4}\bigr].
\end{align}

{\bf Case ${\bf III}_L(x)$}: By independence, we have
\begin{align*}
\E\bigl[{\bf III}_L(x)\bigr] &= P_{}\Big(\big\|Z_{L}^{(o)} + S_L^{(e)}({V}) + R_L^{(o)}(V) + \mu\big\|_{\HH}\leq x \Big) \\ \nonumber &- P_{}\Big(\big\|Z_{L}^{(o)} + Z_L^{(e,o)} + \mu\big\|_{\HH}\leq x \Big),
\end{align*}
hence we may directly appeal to Lemma \ref{lem_ulyanov}. Routine calculations then reveal
\begin{align}
\sup_{x \in \R}\bigl|\E\bigl[{\bf III}_L(x)\bigr]\bigr| \lesssim L^{-1/2}\bigl(1 + \|\mu\|_{\HH}^3\bigr) \E\bigl[\|X_0\|_{\HH}^{3}\bigr].
\end{align}

{\bf Case ${\bf IV}_L(x)$}: One readily verifies
\begin{align*}
\bigl|\tr\bigl({\bf \Lambda}\bigr) - \tr\bigl({\bf \Lambda}^{(o)} + {\bf \Lambda}^{(e)}\bigr) \bigr| \lesssim  L^{-1}\E\bigl[\|X_0\|_{\HH}^{2}\bigr].
\end{align*}
Since we have by independence
\begin{align*}
P_{|\F}\Big(\big\|Z_{L}^{(o)} + Z_L^{(e,o)} + \mu\big\|_{\HH}\leq x \Big) = P_{}\Big(\big\|Z_{L}^{(o)} + Z_L^{(e,o)} + \mu\big\|_{\HH}\leq x \Big),
\end{align*}
an application of Lemma \ref{lem_gauss_hoelder_inf} then yields
\begin{align}
\sup_{x \in \R}{\bf IV}_L(x) \lesssim L^{-1} \E\bigl[\|X_0\|_{\HH}^{2}\bigr].
\end{align}

Since $n \thicksim L$, combining all four bounds completes the proof. As a final remark, let us elaborate on the case where $L \not \in  \N$. In this case, we may have a slightly smaller additional remainder term $\tilde{R}_{L+1}$ in the decomposition
\begin{align}\label{final_decomp}
n^{-1/2}S_n(X) = S_L^{(o)}(\overline{V}) + S_L^{(e)}({V}) + R_L^{(o)}({V}) + \tilde{R}_{L+1},
\end{align}
which we can always add to the last summand, be it even or odd. This just results in more complicated notation, but the proof remains the same.
\end{proof}

\begin{proof}[Proof of Corollary \ref{cor_m_dependence}]
If $L = n/m \in \N$, writing
\begin{align*}
(nm)^{-1/2}S_n(X) = (n/m)^{-1/2}\sum_{l = 1}^L B_l/m, \quad B_l = \sum_{k = (l-1)m+1}^{lm} X_k,
\end{align*}
we may directly apply Theorem \ref{thm_non_lin}. If $L \not \in \N$, we have an additional remainder part, which however does not require any particular different treatment, see also the remark around \eqref{final_decomp}.
\end{proof}

\begin{proof}[Proof of Corollary \ref{cor_geo_dependence}]
By virtue of Theorem \ref{thm_non_lin}, we may almost identically repeat the proof of Theorem 2.1 together with Corollary 2.2 in ~\cite{Jirak_2015_pol}.
\end{proof}

%need lip-continuity of gaussian hilbert space valued random variables, e.g.: eq. 1.2 in: On Berry-Esseen Type Bounds for m-Dependent Random Variables Valued in Certain Banach Spaces

\section*{Acknowledgements}

I would like to thank the Associate Editor and the anonymous Reviewers for the constructive comments and suggestions. The generous help has been of major benefit.

\begin{small}

\end{small}

\end{document}